\documentclass{article}
\usepackage{amssymb,amsmath}
\usepackage{xcolor}
\usepackage{amsthm}
\usepackage[colorlinks=true]{hyperref}
\usepackage{pdfsync}

\newcommand{\ltwoR}{{\ell^2 ({\mathbb R})}}
\newcommand{\ltwoC}{{\ell^2 ({\mathbb C})}}
\newtheorem{theorem}{Theorem}[section]
\newtheorem{lemma}{Lemma}[section]

\newtheorem{corollary}{Corollary}[section]
\newtheorem{proposition}{Proposition}[section]
\theoremstyle{definition}

\theoremstyle{remark}
\newtheorem{remark}{Remark}
\newtheorem{problem}{Problem}%
\newcommand{\Renu}{{\rm Re}\,\nu_k}
\newcommand{\Imnu}{{\rm Im}\,\nu_k}

\newcommand{\norm}[1]{\left\lVert #1\right\rVert}

\renewcommand{\Re}{\mathrm{Re}\,}

\begin{document}

\title{Polynomial Turnpike Property for a Class of Infinite-Dimensional Oscillating Systems}
\date{}

\author{Alexander Zuyev\footnote{Max Planck Institute for Dynamics of Complex Technical Systems, 39106 Magdeburg, Germany}\footnote{Institute of Applied Mathematics and Mechanics, National Academy of Sciences of Ukraine, e-mail: zuyev@mpi-magdeburg.mpg.de.}
\and
Emmanuel Tr\'elat \footnote{Sorbonne Université, Université Paris Cité, CNRS, Inria, Laboratoire Jacques-Louis Lions, LJLL, F-75005 Paris, France, e-mail: emmanuel.trelat@sorbonne-universite.fr.}}

\maketitle

\begin{abstract}
We establish a polynomial turnpike estimate for an optimal control problem consisting of a system of infinitely many controlled oscillators, considered as an
abstract differential equation in a Hilbert space, with a quadratic cost.
Our proof relies on spectral considerations and on the construction of a Riesz basis.
A concrete example is given,
which involves a rotating body-beam system.
To our knowledge, this is the first example of a pointwise turnpike estimate around a steady-state that is polynomial but not exponential.
\end{abstract}


\section{Introduction}
It is well-known that many important classes of infinite-dimensional control systems, which are strongly stabilizable by static feedback laws, do not possess any exponential decay estimate
%(see, e.g.,~\cite{zhang2004polynomial,liu2005characterization,anantharaman2014sharp,sklyar2021polynomial}).
{(see, e.g.,~\cite{zhang2004polynomial, liu2005characterization, anantharaman2014sharp, sklyar2021polynomial, batty2008nonuniform, borichev2010optimal, batty2016fine}).}
This feature differs essentially from the theory of finite-dimensional autonomous systems, where non-exponential (but polynomial) asymptotic stability is a purely nonlinear phenomenon in the case of a neutral linear approximation (see, e.g.,~\cite{grushkovskaya2013asymptotic} and references therein).
The goal of this paper is to characterize the large-time behavior of a class of undamped hyperbolic systems under Riccati-based controllers and compare solutions to static and dynamic optimization problems in the context of the turnpike property.

The concept of turnpike, originated from the work of von Neumann~\cite{neumann1945model} and developed in~\cite{dorfman1958linear} as a part of capital accumulation theory, has a fundamental influence on modern control theory (see, e.g.,~\cite{zaslavski2005turnpike, trelat2018steady, faulwasser2022turnpike} and references therein).
{
Over the last decade, several complementary approaches to turnpike have been developed, including strict dissipativity and economic model predictive control viewpoints \cite{grune2016relation, faulwasser2017economic, faulwasser2017turnpike}, sensitivity-based analyses for PDE-constrained linear-quadratic problems \cite{grune2019sensitivity, grune2020sensitivity}, hyperbolic boundary control settings \cite{gugat2019hyperbolic}, and semilinear PDE control problems \cite{pighin2021semilinear}.
We also mention geometric viewpoints for nonlinear optimal control \cite{sakamoto2021geometric}, shape optimization settings \cite{lance2020shape}, and recent extensions on large-time intervals and time-dependent structures \cite{askovic2024value, zamorano2025almost}.
For broader perspectives, we refer to the recent surveys \cite{geshkovski2022turnpike, trelat2025survey}.
}
In the recent paper~\cite{han2022slow}, an \emph{integral} turnpike estimate was proposed for a linear hyperbolic system with quadratic cost under weak controllability and observability assumptions.
{To the best of our knowledge, we are not aware of a pointwise polynomial turnpike property around a steady-state in the existing literature.}%
%has never been obtained.
{Note that a pointwise non-exponential, called ``linear turnpike'' property, has been established in~\cite{trelat2023linear}, but there, the turnpike set is not a steady-state: it is a nonsteady
trajectory and this is because of the non-steady feature that the exponential estimate is lost. In the present paper, the turnpike set is a singleton.}
A major difficulty is related to the Riesz basis property of the associated infinitesimal generator.

In this paper, we establish pointwise polynomial turnpike estimates for a class of undamped hyperbolic problems with one-dimensional control.

A distinctive feature of our main result is that the turnpike estimate is obtained in a weaker weighted topology $V_{0,-\beta-1}$.
This weakening is intrinsic to the polynomial decay mechanism {in the case of zero spectral abscissa}:
the available mode-by-mode decay estimates naturally propagate through the Hamiltonian variational dynamics only in negatively weighted norms.

The polynomial nature of the turnpike estimate is a direct consequence of a vanishing spectral {abscissa} at high frequencies: the real parts of the relevant eigenvalues satisfy $\Re\sigma_k^\pm\sim\vert b_k\vert\to0$ as $k\to+\infty$.
In this regime, the modal decay $e^{-\Re\nu_k t}$ can be converted into a uniform polynomial bound only at the price of a mode-dependent prefactor, for instance $e^{-\vert b_k\vert t}\leq C_\beta\vert b_k\vert^{-\beta}(t+1)^{-\beta}$ for any $\beta>0$.
This loss is precisely compensated by working in the weak weighted topology $V_{0,-\beta-1}$ (equivalently, in the weighted sequence norm $H^b_{-(\beta+1)}$ introduced later on), which is therefore the natural functional setting for a pointwise polynomial turnpike statement.
In particular, on any fixed finite-dimensional spectral truncation, the same estimate yields a polynomial turnpike in the strong norm, with constants independent of the horizon.

Besides Assumptions \eqref{A1}--\eqref{A2} introduced hereafter, we use the gap-summability condition \eqref{A3} and the weighted Bari-type condition \eqref{A5}.
If \eqref{A5} holds for some $\rho>1$, then the polynomial turnpike exponent is $\beta=\rho-1$.
More precisely, for the free terminal point linear-quadratic problem, the optimal deviation from the steady optimal quadruple is ${\mathcal O}\big((t+1)^{-\beta}+(T-t+1)^{-\beta}\big)$ in $V_{0,-\beta-1}$, uniformly with respect to the {time} horizon $T$.

The article is structured as follows.
Section~\ref{sec_modal} considers static and dynamic optimization problems with quadratic costs for a linear infinite-dimensional system with a single input. The main result, presented in Theorem~\ref{thm_turnpike_coordinates}, is a polynomial turnpike estimate for the variational system. The proof of this estimate relies on the spectral theory of linear operators and auxiliary results, which are summarized in Section~\ref{sec_spectrum}.
In Section~\ref{sec_Riesz}, we construct a Riesz basis using linear combinations of the eigenvectors of the infinitesimal generator for the variational system.
This construction is applied in the proof of Proposition~\ref{thm_turnpike}, presented in Section~\ref{sec_proof}.
The obtained analytical results are illustrated with an example of a rotating flexible beam model in Section~\ref{sec_beam}.

\section{Oscillating system with modal coordinates}\label{sec_modal}
Consider an infinite system of controlled oscillators
\begin{equation}
\begin{aligned}
&\dot \xi_ k(t) = \omega_k \eta_k(t),\\
&\dot \eta_k(t) = -\omega_k \xi_k(t) + b_k u(t),\quad k\in {\mathbb N},
\end{aligned}
\label{cs_coord}
\end{equation}
where the column vectors $\xi=(\xi_1,\xi_2, ...)^\top\in\ltwoR$ and $\eta=(\eta_1,\eta_2, ...)^\top\in\ltwoR$ play the role of generalized coordinates and momenta, respectively, and the real Hilbert space $\ltwoR$ is equipped with the inner product
$
\left< \xi,\eta \right>_\ltwoR  = \sum_{k=1}^\infty \xi_k \eta_k
$.
Throughout the text, the set of positive integer numbers is denoted by $\mathbb N=\{1,2,...\}$.
The control system~\eqref{cs_coord} is considered as a mathematical model for vibrating flexible structures, such as strings, beams, and plates (see, e.g.,~\cite[Chap.~3]{zuyev2015partial}).
We assume that the real parameters $b_k$ and $\omega_k$ satisfy:
\begin{equation}\label{A1}
b=(b_1,b_2,...)^\top \in \ltwoR\;\text{and} \; b_k\neq 0\;\text{for all}\; k\in{\mathbb N},
\tag{$A1$}
\end{equation}
\begin{equation}\label{A2}
\omega_1>0,\; \omega_{k+1} - \omega_{k}\geq \omega_*>0\quad \text{for all}\;\; k\in{\mathbb N}.
\tag{$A2$}
\end{equation}
For some of the technical results presented below, we will also consider some of the following assumptions:
{
\begin{equation}\label{A3}
S_k = \sum_{j\neq k} \frac{1}{(\omega_k-\omega_j)^2}\leq {\bar C}<\infty\quad\text{for all}\;k\in{\mathbb N},
\tag{$A3$}
\end{equation}}
{
\begin{equation}\label{A4}
\sum_{k=1}^\infty \sum_{j\neq k} \frac{b_j^2}{(\omega_j-\omega_k)^2} < \infty,
\tag{$A4$}
\end{equation}}
{
\begin{equation}\label{A5}
\sum_{k\geq1}\frac{1}{|b_k|^{2\rho}}
\sum_{j\neq k}\frac{|b_j|^{2(1+\rho)}}{(\omega_j-\omega_k)^2}<\infty\quad\text{for some}\; \rho>1,
\tag{$A5$}
\end{equation}}
{
\begin{equation}\label{A6}
\liminf_{k\to \infty}\left(\vert b_k \vert \omega_k^\alpha \right) >0\quad\text{for some}\; \alpha>0.
\tag{$A6$}
\end{equation}}
It is clear that~\eqref{A3} together with \eqref{A1} implies~\eqref{A4}.

\begin{remark}\label{rem_paramasymptotics}
If we assume that $\omega_k\asymp k^p$ with $p\geq 1$ and $|b_k|\asymp \omega_k^{-\alpha}$ for some $\alpha>0$,
then condition~\eqref{A6} holds.
Moreover, a standard near/far decomposition together with the gap condition \eqref{A2} shows that \eqref{A5} holds provided
\begin{equation}\label{alpha_rho_p}
2\alpha\rho<2-\frac1p.
\end{equation}
In Section~\ref{sec_beam}, we will present a rotating beam example with $p=2$ and $\alpha=\tfrac12$, so that \eqref{A5} holds for such a example for all $\rho<\tfrac32$.
\end{remark}

The system~\eqref{cs_coord} is written in the operator form as
\begin{equation}
\dot x(t) = Ax(t) + Bu(t),\;\; x(t)\in H = \ltwoR \times \ltwoR,\; u(t)\in \mathbb R,
\label{cs_op}
\end{equation}
where
$$
\begin{aligned}
&x= \begin{pmatrix}\xi \\ \eta\end{pmatrix},\; A= \begin{pmatrix}0 & \Omega \\ -\Omega & 0 \end{pmatrix},\; B= \begin{pmatrix}0 \\ b\end{pmatrix},\\
& \Omega = {\rm diag}(\omega_1,\omega_2,...).
\end{aligned}
$$
We consider the interpolation spaces {$V_{p,q}$} with $p,q\in \mathbb R$ (cf.~\cite{triebel,jacob2006controllability}), defined as
\begin{equation}\label{interp_space}
{V_{p,q}=\left\{x=\begin{pmatrix}\xi \\ \eta\end{pmatrix} \,\vert \, \sum_{k=1}^\infty  \frac{\omega_k^{2 p}( \xi_k^2 +\eta_k^2 )}{|b_k|^{2q}} <\infty \right\}},
\end{equation}
where {$V_{p,q}$} is the Hilbert with respect to the inner product
$$\left<\begin{pmatrix}\xi \\ \eta\end{pmatrix},\begin{pmatrix}\tilde\xi \\ \tilde\eta\end{pmatrix}\right>_{V_{p,q}}=
\sum_{k=1}^\infty \frac{\omega_k^{2 p} ( \xi_k \tilde \xi_k +\eta_k \tilde \eta_k)}{{|b_k|^{2q}}}.
$$
Identifying $H={V_{0,0}}$ with its dual, the dual {of $V_{p,q}$ with respect to the pivot space $H$ is $V_{-p,-q}$.}
Moreover, we have $V_{0,1} \subset H \subset {V_{0,-1}}$ with continuous and dense inclusions.
This follows from the property that {$b\in \ell^2$, hence $b\in \ell^\infty$ and $\|x\|_H^2=\sum_k |b_k|^2\frac{\xi_k^2+\eta_k^2}{|b_k|^2}\leq \|b\|_{\ell^\infty}^2\|x\|_{V_{0,1}}^2$.}

Note that, under Assumption~\eqref{A6},
there exists a constant $C>0$ such that $|b_k|^{-2} \leq C^2 \omega_k^{2\alpha}$, which leads to $\|x\|_{V_{0,1}} \leq C \|x\|_{V_{\alpha,0}}$.
Therefore, $V_{\alpha,0} \subset {V_{0,1}}$, and the inclusion is continuous and dense. So, under~\eqref{A6}, we have
$$
V_{\alpha,0}
\hookrightarrow {V_{0,1}} \hookrightarrow H \hookrightarrow {V_{0,-1}} \hookrightarrow V_{-\alpha,0}.
$$

Note that the operator $A: D(A)\to H$ is skew-adjoint and densely defined under Assumption \eqref{A2} with
$D(A)=V_{1,0}$, and $V_{-1,0}=D(A^*)'$.
It follows from the Lumer--Phillips theorem (see, e.g.,~\cite{engel2000one,tucsnak2009observation}) that $A$ is the infinitesimal generator of a unitary group $\{e^{tA}\}_{t\in\mathbb R}$ of bounded linear operators on $H$.
%For a given time horizon $T$, we define the class of admissible controls for system~\eqref{cs_op} as ${\cal U}_T=L^\infty(0,T)$.
Given any $T>0$, any $x^0\in H$, and any control $u(\cdot)\in L^2(0,T)$,
there exists a unique solution
$x(\cdot) \in C^0([0,T];H) \cap H^1((0,T);{V_{-1,0}})$ of~\eqref{cs_op} (see, e.g., \cite[Part~II, Section~4.2, Proposition~4.12]{trelat2023control}) such that $x(0)=x^0$, and we have
\begin{equation}\label{mildsol}
x(t) = e^{tA} x^0 + \int_0^t e^{(t-s)A} Bu(s)\,ds,\quad t\in [0,T].
\end{equation}
%Here, $D(A^*)'$ is the dual of $D(A^*)$ with respect to the pivot space $H$:
%$$
%D(A^*) = \left\{x \,|\, \sum_{k=1}^\infty \omega_k^{-2} |\xi_k^2+\eta_k^2| < %\infty\right\}.
%$$
%Recall also that the operator $A$ is naturally extended to an operator (still denoted) $A$: $H \rightarrow D(A^*)'$, of dense domain $H$,
%generating on $D(A^*)'{=H_{-1}}$ the extension of the semigroup $\{e^{tA}\}$ to ${H_{-1}}$ (see, e.g.,~\cite{engel2000one,tucsnak2009observation,trelat2023control}).
%
Given any $p,q \in\mathbb R$, depending on the signs of $p$ and $q$, the operator $A$ is naturally restricted or extended to an operator (still denoted)
$A$
{ on $V_{p,q}$, with domain $V_{p+1,q}$},
generating on $V_{p,q}$ the semigroup $\{e^{tA}\}_{t\geq 0}$ that is either the restriction or the extension of the
semigroup on $H$ (see~\cite[Section 4.1.3]{trelat2023control}).

Assumptions~\eqref{A1} and \eqref{A2} are crucial for ensuring controllability properties of system~\eqref{cs_op}.
Recall that, given $x^0\in H$ and $T>0$, the reachable set of system~\eqref{cs_op} at time $T$ is defined as ${\cal R}_T (x^0) = \{x(T)\,|\, u(\cdot)\in L^2(0,T)\}$, with the solutions $x(t)$ represented by formula~\eqref{mildsol},
and ${\cal R}(x^0) = \bigcup_{T>0}{\cal R}_T (x^0)$.
We summarize the controllability conditions in the following proposition.

\begin{proposition}\label{lem_approxcontr}
Let Assumptions~\eqref{A1} and \eqref{A2} hold, and let $T>T_0$, where
\begin{equation}\label{T0}
T_0 = \frac{2\pi}{\omega_*}.
\end{equation}
Then, system~\eqref{cs_op} is exactly controllable in time $T$ in the Hilbert space ${V_{0,1}}$,
and we have  ${\cal R}_T (x^0)= {V_{0,1}}$ for each $x^0\in {V_{0,1}}$,
i.e., ${V_{0,1}}$ is the largest Hilbert subspace of $H$ in which the system is exactly controllable in time $T$.
Moreover, any solution of the control system starting at a point of ${V_{0,1}}$ remains in ${V_{0,1}}$, i.e.,
we have also well-posedness in ${V_{0,1}}$.
\end{proposition}
\begin{proof}
{It is well-known that, by duality (see, e.g.,~\cite[Theorem~5.30]{trelat2023control}), system~\eqref{cs_op} is exactly controllable in ${V_{0,1}}$ in time $T>0$ if and only if there exists a $C_T>0$ such that the following observability inequality holds:
\begin{equation}
\int_0^T |B^* e^{(T-t)A^*} x|^2 dt \geq C_T \|x\|^2_{{V_{0,-1}}}\quad\text{for all}\; x\in H.
\label{obseq}
\end{equation}
For $x=\begin{pmatrix} \xi \\ \eta \end{pmatrix}\in H$, the coordinate representation of~\eqref{obseq} takes the form
\begin{equation}
\int_0^T \Bigl| \sum_{k\in {\mathbb Z}\setminus \{0\}} c_k e^{i\omega_k t}\Bigr|^2 dt \geq C_T \sum_{k\in\mathbb N} b_k^2 (\xi_k^2+\eta_k^2),
\label{obseq_coord}
\end{equation}
where $\omega_{-k}:=\omega_k$, $c_k:=\frac{b_k(\eta_k - i\xi_k)}{2}$, and $c_{-k}:=\frac{b_k(\eta_k + i\xi_k)}{2}$ for $k\in\mathbb N$.}

{Haraux's generalization of Ingham's inequality~\cite[Chap.~2]{haraux1991systemes} and~\cite[Theorem~4.6]{komornik2005fourier} implies the existence of a constant $C>0$ such that
\begin{equation}
\int_0^T \bigl| \sum_{k\in {\mathbb Z}\setminus \{0\}} c_k e^{i\omega_k t}\bigr|^2 dt \geq C \sum_{k\in{\mathbb Z}\setminus \{0\}} |c_k|^2,
\label{ingham}
\end{equation}
for all square-summable complex sequences $\{c_k\}_{k\in{\mathbb Z}\setminus \{0\}}$, provided that $\inf_{k\neq l} |\omega_k-\omega_l|>0$ and $T>2\pi/{\tilde \omega}$ with
\begin{equation}
\label{sep_condition}
\tilde\omega:=\sup_{K\subset {\mathbb Z}\setminus \{0\}} \inf_{\scriptsize\begin{matrix}k,l\in {\mathbb Z}\setminus (K\cup \{0\}) \\ k\neq l\end{matrix}} |\omega_k - \omega_l|,
\end{equation}
where $K$ runs over the finite subsets of ${\mathbb Z}\setminus \{0\}$. Under Assumption~\eqref{A2}, $\tilde\omega\geq \omega_*$, so that inequality~\eqref{ingham} holds for each $T>T_0$. By combining together~\eqref{obseq_coord} and~\eqref{ingham}, we conclude that the observability inequality~\eqref{obseq_coord} holds.
Moreover, ${V_{0,1}}$ is the largest possible space in which one has controllability, because
the Ingham inequality is ``optimal'', in the sense that the converse inquality holds true.
Actually, the converse inequality is what is referred to in the literature as ``admissibility'',
yielding that solutions of the control system starting in ${V_{0,1}}$, remain in ${V_{0,1}}$. Here, we prove, as
well, that we have the ``converse inequality~\eqref{obseq}'' (and thus, an equivalence of norms
(see~\cite[Proposition 5.13]{trelat2023control}). This shows that we have both controllability and
admissibility in ${V_{0,1}}$ (or in other words, ${V_{0,1}}$ is the Hilbert space given by HUM).}
\end{proof}

\begin{remark}\label{rem_controllability}
Since ${V_{0,1}}$ is dense in $H$, the conditions of Proposition~\ref{lem_approxcontr} imply that system~\eqref{cs_op} is approximately controllable in $H$ for any time $T>T_0$.
However, this does not guarantee exact controllability in $H$.
\end{remark}

%\begin{lemma}\label{lem_uncontrollable}
%System~\eqref{cs_op} is not exactly controllable in $H$,
%and ${\cal R}(x^0)\supsetneq H$ for each $x^0\in H$.
%\end{lemma}
%The proof of this lemma is provided in Section~\ref{sec_proof}.
%

We will consider two types of optimal control problems for the system~\eqref{cs_op}.

\begin{problem}{\em (Static optimization problem)}\label{problem_stat}
Given a vector $\bar x = \begin{pmatrix}\bar \xi \\ \bar\eta \end{pmatrix}\in {V_{0,1}}$ and a real constant $\bar u$,
minimize the cost
$$
J_s(x,u) = \|x -\bar x\|^2_H +  \vert u -\bar u \vert^2
$$
under the constraints
\begin{equation}\label{static_constraints}
x = \begin{pmatrix}\xi \\ \eta\end{pmatrix}\in H,\;
Ax + B u = 0.
\end{equation}
(Here, it is understood that $A$ stands for the extension $A:H\rightarrow V_{-1,0}$.)
\end{problem}

\begin{problem}{\em (Dynamic optimization problem)}\label{problem_dyn}
Given a time horizon $T>0$, {initial value $x^0 \in {V_{0,1}}$}, a vector $\bar x = \begin{pmatrix}\bar \xi \\ \bar\eta \end{pmatrix}\in H$, and a real constant $\bar u$,
minimize the cost
$$
J_d(x(\cdot),u(\cdot)) = \int_0^T \left( \|x(t)-\bar x\|^2_H + \vert u(t)-\bar u \vert^2 \right)dt
$$
over all possible solutions $x(\cdot)$ of~\eqref{cs_op} corresponding to admissible controls $u(\cdot)\in L^2(0,T)$ and satisfying
{$x(0)=x^0$}.
%This admissible set is nonempty by Proposition~\ref{lem_approxcontr}.
\end{problem}

The constrained static minimization Problem~\ref{problem_stat} can be treated by introducing the Lagrangian
$$
L_s (x,\lambda,\mu,u) = \frac{1}{2} J_s(x,u) - \sum_{k=1}^\infty \left( \lambda_k \omega_k\eta_k + \mu_k (-\omega_k \xi_k + b_k u) \right),
$$
where $\lambda_k$ and $\mu_k$ are the Lagrange multipliers corresponding to the coordinate components of~\eqref{static_constraints}.
As the functional $J_s$ is strictly convex and coercive on $H\times{\mathbb R}$, there exists a unique minimizer $(\hat x,\hat u)$ of $J_s$ satisfying~\eqref{static_constraints}, associated with a unique Lagrange multiplier $(\hat \lambda,\hat \mu)$ such that
$\hat x\in V_{1,0}$, $(\hat\lambda^\top,\hat \mu^\top)^\top\in V_{1,0}$, and
\begin{equation}\label{static_minimizer}
\begin{aligned}
&\hat x-\bar x = A^* \hat\Lambda,\;
\hat u-\bar u = B^* \hat\Lambda,\; \hat\Lambda=\begin{pmatrix}\hat \lambda \\ \hat \mu \end{pmatrix},\\
&
\hat x=\begin{pmatrix}\hat \xi \\ \hat\eta \end{pmatrix},\; \hat \xi_k = \bar \xi_k - \omega_k \hat \mu_k=\frac{b_k \hat u}{\omega_k},\; \hat \eta_k = \bar \eta_k + \omega_k \hat \lambda_k=0,\\
&\hat \mu_k = \frac{1}{\omega_k}\left( {\bar \xi}_k -\frac{b_k \hat u}{\omega_k}\right),\; \hat \lambda_k = -\frac{\bar \eta_k}{\omega_k},\\
& \hat u = \bar u + \sum_{k=1}^\infty b_k \hat \mu_k = \left(1+\sum_{k=1}^\infty \frac{b_k^2}{\omega_k^2}\right)^{-1}\cdot \left(\bar u+ \sum_{k=1}^\infty  \frac{b_k \bar \xi_k}{\omega_k} \right).
\end{aligned}
\end{equation}

Let us prove that $\hat x \in {V_{0,1}}$.
Indeed, we have $\hat x = (\hat\xi^\top, \hat\eta^\top)^\top$ with $\hat\eta=0$ and
$\hat\xi = \hat u \Omega^{-1} b$, hence $\hat x \in {V_{0,1}}$ by definition of ${V_{0,1}}$.
Besides, $A\hat\Lambda = \bar x - \hat x \in {V_{0,1}}$ because we have assumed that $\bar x\in {V_{0,1}}$.
The claim is proved.

The existence of solutions to Problem~\ref{problem_dyn} follows, e.g., from Theorem~5.10 in~\cite[Chap.~3]{liyong}.
To justify the applicability of the Pontryagin maximum principle for solving this problem, we refer to the finite codimensionality condition outlined in~\cite{li1991necessary}.
%Under Assumption \eqref{A5}, it is easy to see that
%$H_\alpha \subset {\cal C}$.
%
Under the assumptions of Proposition~\ref{lem_approxcontr}, system~\eqref{cs_op} is exactly controllable in ${V_{0,1}}$ in any time $T>T_0$.
Specifically, if
%\begin{equation}\label{optim_sufficient}
$( \hat x(\cdot), \hat u(\cdot) )$ %\;\text{
is optimal for Problem~\ref{problem_dyn}, $T>T_0$, $x^0\in V_{0,1}$,
then, since system~\eqref{cs_op} is exactly controllable in the Hilbert space ${V_{0,1}}$,
the finite codimension condition of~\cite{li1991necessary} is satisfied for this Hilbert space ${V_{0,1}}$.
Therefore, the pair $( \hat x(t), \hat u(t) )$ satisfies the Pontryagin maximum principle (cf.~\cite[Theorem~2.3]{li1991necessary})
with some adjoint variables $(\lambda(t),\mu(t))$
in the dual ${V_{0,-1}}$ of ${V_{0,1}}$
for $t\in [0,T]$.

To find the optimal solutions, we write the Hamiltonian
$$
H(x,\lambda,\mu,u)  = \sum_{k=1}^\infty \left( \lambda_k \omega_k \eta_k - \mu_k\omega_k\xi_k + \mu_k b_k u\right) - \frac{1}{2}\left(
\|x -\bar x\|^2_H   + \vert u -\bar u \vert^2
\right)
$$
and apply the Pontryagin maximum principle (PMP).
{Since $x(T)$ is free and there is no terminal cost in Problem~\ref{problem_dyn}, the PMP yields the terminal condition $\lambda(T)=0$ and $\mu(T)=0$,}
and the adjoint system of equations
with respect to {$\Lambda(t)=\begin{pmatrix}\lambda(t)\\ \mu(t)\end{pmatrix}$}
takes the form
\begin{equation}\label{sys_adjoint}
\dot \Lambda(t) = A \Lambda(t) + x(t)-\bar x,\; \Lambda(T)=0.
\end{equation}
The extremal control is
\begin{equation}\label{u_extremal}
u(t) = \bar u + \sum_{k=1}^\infty b_k \mu_k(t).
\end{equation}

Solving problem~\eqref{sys_adjoint} backward gives
$$
\Lambda(t)=\int_t^T e^{(t-s)A}\,(\bar x - x(s))\,ds.
$$
Problem~\ref{problem_dyn} is formulated under the standing assumption $\bar x\in H$.
Under this assumption, noting that $e^{tA}$ is unitary on $H$, the representation formula above implies $\Lambda(t)\in H$ for all $t\in[0,T]$ whenever $x(\cdot)\in L^1((0,T);H)$.

In several places below, we will use the stronger statement $\Lambda(t)\in V_{0,1}$.
A sufficient condition for this is $x^0\in V_{0,1}$ and $\bar x\in V_{0,1}$, in which case the state satisfies $x(t)\in V_{0,1}$ and the same representation formula holds in $V_{0,1}$.
This additional assumption is  invoked only when we explicitly use $V_{0,1}$-based estimates or when we wish to emphasize the $b$-weighted regularity of extremals.

%Since $e^{tA}$ is unitary on $H$, we obtain $\Lambda\in C([0,T];H)$ as soon as $x-\bar x\in L^1([0,T];H)$.
%Moreover, if we assume $x^0\in V_{0,1}$ and $\bar x\in V_{0,1}$, then Proposition~\ref{lem_approxcontr} gives $x(\cdot)\in C([0,T];V_{0,1})$, hence $x-\bar x\in L^1([0,T];V_{0,1})$.
Since $e^{tA}$ acts mode-by-mode as a rotation, it is an isometry on any $V_{p,q}$ (in particular on $V_{0,1}$), and we get
$\Lambda(t)\in V_{0,1}$ and $\|\Lambda(t)\|_{{V_{0,1}}}\leq \int_t^T \|x(s)-\bar x\|_{{V_{0,1}}}\,ds$ for every $t\in[0,T]$.

Next, let us show well-definiteness of~\eqref{u_extremal} and $u\in L^2(0,T)$.
From $\Lambda\in C([0,T];H)$ we have $\mu(t)\in\ell^2$ for all $t$, hence the series in~\eqref{u_extremal} is meaningful and
$$
|u(t)-\bar u|=\Big|\sum_{k\geq 1}b_k\mu_k(t)\Big|
\leq \|b\|_{\ell^2}\,\|\mu(t)\|_{\ell^2}.
$$
Since $\mu\in C([0,T];\ell^2)\subset L^2([0,T];\ell^2)$, it follows that $u\in L^2(0,T)$ and
$$
\|u-\bar u\|_{L^2(0,T)}\leq \|b\|_{\ell^2}\,\|\mu\|_{L^2(0,T;\ell^2)}<\infty.
$$

By introducing deviations from the steady-state solution~\eqref{static_minimizer}
and substituting the formulas
\begin{equation}\label{delta_notation}
\begin{matrix}
\delta \xi(t) = \xi(t) - \hat \xi,\\
\delta \eta(t) = \eta(t) - \hat \eta,\\
\delta \lambda(t) = \lambda(t) - \hat \lambda,\\
\delta \mu(t) = \mu(t) - \hat \mu,
\end{matrix}
\qquad
\Delta(t) = \left(\begin{matrix} \delta \xi(t) \\ \delta \eta(t) \\ \delta \lambda(t) \\ \delta \mu(t) \end{matrix}\right)
\end{equation}
into~\eqref{cs_coord},~\eqref{sys_adjoint} with control~\eqref{u_extremal}, we obtain the (so-called variational) system of differential equations:
\begin{equation}\label{delta_coord}
\begin{aligned}
& \dot{\delta \xi_k}(t) = \omega_k \delta \eta_k(t),\; \dot{\delta \eta_k}(t) = -\omega_k \delta \xi_k(t) + b_k \sum_{j=1}^\infty b_j \delta \mu_j(t),\\
& \dot{\delta \lambda_k}(t) = \omega_k \delta \mu_k(t) + \delta \xi_k(t),\\
& \dot{\delta \mu_k}(t) = -\omega_k \delta \lambda_k(t) + \delta \eta_k(t),\; k\in \mathbb N.
\end{aligned}
\end{equation}

We then rewrite system~\eqref{delta_coord} in the abstract form as
\begin{equation}\label{delta_op_small}
\dot \Delta(t) = \tilde A \Delta(t),\quad \Delta(t) \in \tilde X = {V_{0,1}}\times {V_{0,1}},
\end{equation}
$$
\tilde A = \begin{pmatrix} 0 & \Omega & 0 & 0\\
-\Omega & 0 & 0 & b b^* \\
I & 0 & 0 & \Omega \\
0 & I & -\Omega & 0\end{pmatrix},
$$
with $I$ standing for the identity operator.
%, and $\|\Delta(t)\|_X=\left( \|\delta \xi(t) \|^2 + \|\delta \eta(t) \|^2 + \|\delta \lambda(t) \|^2 + \|\delta \mu(t) \|^2  \right)^{1/2}$.

The main result of this paper is a pointwise polynomial
turnpike estimate, which is formulated as follows.

\begin{theorem}\label{thm_turnpike_coordinates}
Assume \eqref{A1}, \eqref{A2},~\eqref{A3}, and fix $\beta>0$ such that \eqref{A5} holds with $\rho=\beta+1$.
For any $T>0$, let $(x,u)$ be the unique optimal solution of Problem~\ref{problem_dyn}, and let $\Lambda=(\lambda,\mu)$ be the associated adjoint variables.
Let $(\hat x,\hat u,\hat\Lambda)$ denote the steady optimal triple of the static Problem~\ref{problem_stat}.
Then there exists a constant $C>0$, independent of $T$ and of the solution, such that for every $t\in[0,T]$,
\begin{multline}\label{eq:main_turnpike}
\Vert x(t)-\hat x\Vert_{V_{0,-\beta-1}} + \Vert \Lambda(t) - \hat\Lambda \Vert_{V_{0,-\beta-1}} \\
\;\leq C \left( \Vert x(0)-\hat x\Vert_{V_{0,1}} + \Vert\hat\Lambda\Vert_{H} \right) \left(\frac{1}{(t+1)^\beta}+\frac{1}{(T-t+1)^\beta}\right).
\end{multline}
\end{theorem}

This means that the optimal trajectory stays polynomially close (in $V_{0,-\beta-1}$) to the steady optimal state
on any interior sub-interval $[\varepsilon,T-\varepsilon]$.

Theorem~\ref{thm_turnpike_coordinates} will be proved in Section~\ref{sec_proof}, using auxiliary Lemmas~\ref{lem_spectrum_balls}--\ref{lem_coshbound} and Propositions~\ref{thm_turnpike}--\ref{prop_shoot},
which are presented below.

\section{Proof of Theorem~\ref{thm_turnpike_coordinates}}\label{sec_proof}

For greater generality in describing the large-time behavior of the variational system~\eqref{delta_op_small},
we consider its state $\Delta(t)$ in the complex-valued Hilbert space $X=(\ltwoC)^4$, where
$\ltwoC$ is the Hilbert space
of square-summable complex sequences with the inner product defined as
$
\left< z,\zeta \right>_\ltwoC = \sum_{k=1}^\infty z_k {\overline \zeta}_k
$.
As a result, the ``complexified'' system~\eqref{delta_op_small} takes the form
\begin{equation}\label{delta_op}
\dot \Delta(t) = \tilde A \Delta(t),\quad \Delta(t) \in X = (\ltwoC)^4,
\end{equation}
with $\tilde A:D(\tilde A)\to X$ represented in its block-matrix form as
$$
\tilde A =
\begin{pmatrix} A & B B^* \\ I & -A^*
\end{pmatrix}.
$$
We also perform the following transformation in system~\eqref{delta_op} with  complex coefficients:
\begin{equation}\label{complex_var}
z = \delta \xi + i \delta\eta,\; \zeta = \delta \xi - i \delta\eta,\; p = \delta \lambda + i \delta \mu,\; q= \delta \lambda - i \delta \mu.
\end{equation}
With this transformation, system~\eqref{delta_op} takes the form:
\begin{equation}\label{theta_c}
\dot \theta(t)= {\cal A} \theta(t),\quad \theta(t) \in X=\left(\ltwoC\right)^4,
\end{equation}
where
\begin{equation}
\theta = \begin{pmatrix} z \\ \zeta \\ p \\ q
\end{pmatrix},\;
{\cal A} = \begin{pmatrix} -i\Omega & 0 & \frac{1}{2}bb^* & -\frac{1}{2}bb^* \\
0 & i\Omega & -\frac{1}{2}bb^* & \frac{1}{2}bb^* \\
I & 0 & -i\Omega & 0 \\
0 & I & 0 & i\Omega
\end{pmatrix}.\label{A_op}
\end{equation}
Note that the change of variables~\eqref{complex_var} is defined on the whole complex space $X$ by means of the linear operator $\pi_1:X\to X$, such that
\begin{equation}\label{pi1op}
X\ni \Delta =\begin{pmatrix} \delta \xi \\ \delta \eta \\ \delta \lambda \\ \delta \mu \end{pmatrix} \mapsto \pi_1 \Delta = \begin{pmatrix}\delta \xi + i \delta\eta \\ \delta \xi - i \delta\eta \\ \delta \lambda + i \delta \mu \\ \delta \lambda - i \delta \mu\end{pmatrix}\in X,
\end{equation}
and the inverse operator $\pi_1^{-1}:X\to X$ is defined by the rule:
\begin{equation}\label{pi1inv}
X\ni\theta = \begin{pmatrix} z \\ \zeta \\ p \\ q
\end{pmatrix} \mapsto \pi_1^{-1} \theta = \frac12 \begin{pmatrix} z+\zeta \\ i(\zeta - z) \\ p+q\\ i (q-p)\end{pmatrix}\in X.
\end{equation}
In what follows, we will omit subscripts on $\|\cdot\|$ when the context is clear.
Using the triangle inequality, we conclude that the above operators are bounded:
\begin{equation}\label{pi1norm}
\|\pi_1\| \leq 2,\; \|\pi_1^{-1}\| \leq 1.
\end{equation}

\subsection{Analysis of the spectrum}\label{sec_spectrum}
Let us perform the spectral analysis for the operator $\cal A$. First, we compute its eigenelements. Assume that
\begin{equation}\label{spectral_pb}
{\cal A} \theta = \sigma \theta,\; \theta \in D({\cal A})\subset X.
\end{equation}
This operator equation can be written with respect to the coordinates of $(z,\zeta,p,q)$ indexed by $k$ as
\begin{equation}\label{spectral_k}
\begin{aligned}
& -i \omega_k z_k + b_k \phi = \sigma z_k, \\
& i \omega_k \zeta_k - b_k \phi = \sigma \zeta_k, \\
& z_k - i\omega_k p_k = \sigma p_k,\\
& \zeta_k + i\omega_k q_k = \sigma q_k,\; k\in \mathbb N,
\end{aligned}
\end{equation}
where
\begin{equation}\label{phi_form}
\phi = \frac{1}{2} \sum_{j=1}^\infty b_j(p_j-q_j) = \frac{1}{2}b^* (p-q).
\end{equation}

Under~\eqref{A1}, the system of equations~\eqref{spectral_k} has only the trivial solution if $\sigma = \pm i \omega_k$ for some $k\in\mathbb N$.
Hence, for any $k\in\mathbb N$, $\sigma = \pm i \omega_k$ is not an eigenvalue of $\cal A$.
Assuming that $\sigma\notin \{\pm i\omega_1,\pm i\omega_2,...\}$, we can solve~\eqref{spectral_pb} with respect to the components of $\theta$
parameterized by $\phi$:
\begin{equation}\label{theta_phi}
\begin{aligned}
& z = (\sigma I + i\Omega)^{-1}b\phi,\; \zeta = - (\sigma I - i\Omega)^{-1}b\phi,\\
& p = (\sigma I + i\Omega)^{-1}z,\;q = (\sigma I - i\Omega)^{-1}\zeta.
\end{aligned}
\end{equation}
The substitution of these formulas into~\eqref{phi_form} yields the solvability condition of~\eqref{spectral_pb}:
\begin{equation}\label{phi_relation}
\phi = \frac{1}{2}b^* \left( (\sigma I + i\Omega)^{-2} +(\sigma I - i\Omega)^{-2}  \right) b \phi.
\end{equation}
When $\phi =0$, the only solution of~\eqref{spectral_pb} is $\theta=0$. Hence, the nontrivial solutions of~\eqref{spectral_pb} correspond to $\phi\neq 0$; in this case we obtain the \emph{characteristic equation} of problem~\eqref{spectral_pb} from~\eqref{phi_relation}:
\begin{equation}\label{chareq}
\sum_{k=1}^\infty b_k^2 \left( \frac{1}{(\sigma + i \omega_k)^2} +  \frac{1}{(\sigma - i \omega_k)^2} \right) = 2,
\end{equation}
or, equivalently:
$$
\sum_{k=1}^\infty \frac{b_k^2(\sigma^2 - \omega_k^2)}{(\sigma^2+\omega_k^2)^2} = 1.
$$
The point spectrum $\sigma_p({\cal A})$ of the operator $\cal A$ is defined by the solutions of~\eqref{chareq}.

\begin{lemma}\label{lem_spectrum_balls}
The point spectrum $\sigma_p({\cal A})$ of $\cal A$ has the following properties:
\begin{enumerate}
\item[1)] $\sigma_p({\cal A})$ is Hamiltonian: if $\sigma$ is an eigenvalue of $\cal A$, then so are $-\sigma$, $\overline{\sigma}$, and $-\overline{\sigma}$;
\item[2)] $\sigma_p({\cal A})\cap i{\mathbb R} = \emptyset$, i.e., the spectrum is hyperbolic;
\item[3)] $\sigma_p({\cal A}) \subset \cup_{k=1}^\infty({\cal B}_k^+ \cup {\cal B}_k^-)$, where
\end{enumerate}
$$
{\cal B}_k^+ = \{\sigma \in {\mathbb C}\, \vert \, \vert\sigma - i\omega_k\vert \leq \|b\|_\ltwoR\},
\qquad
{\cal B}_k^- = \{\sigma \in {\mathbb C}\, \vert \, \vert\sigma + i\omega_k\vert \leq \|b\|_\ltwoR\}.
$$
\end{lemma}
\begin{proof}
The characteristic equation~\eqref{chareq} is invariant under changing $\sigma$ to $-\sigma$ and $\sigma$ to $\overline{\sigma}$, which proves the first property.
This is reminiscent from the fact that the matrix $\Delta$ is Hamiltonian (see~\cite[Lemma 2]{trelat2015turnpike}).
To prove the second one, we recall that it has been already noted that $\pm i \omega_k$ are not eigenvalues of $\cal A$.
Moreover, the left-hand side of~\eqref{chareq} is negative if $\sigma=i\omega$, $\omega\in \mathbb R$;
therefore, the characteristic equation has no purely imaginary roots.

It remains to prove the third property. For this purpose, we estimate the left-hand side of~\eqref{chareq} by using Hölder's inequality:
\begin{equation}
\left\vert \sum_{k=1}^\infty  \left( \frac{b_k^2}{(\sigma + i \omega_k)^2} +  \frac{b_k^2}{(\sigma - i \omega_k)^2} \right) \right\vert 
\leq 2 \|b\|^2_\ltwoR
\cdot \sup_{j\in\mathbb N} \left\{ \frac{1}{\vert\sigma + i \omega_k\vert^2}, \frac{1}{\vert\sigma - i \omega_k\vert^2}\right\}.
\label{sum_est}
\end{equation}
If~\eqref{chareq} holds for some $\sigma\in\mathbb C$, then~\eqref{sum_est} implies
$$
\sup_{j\in\mathbb N} \left\{ \frac{1}{\vert\sigma + i \omega_k\vert^2}, \frac{1}{\vert\sigma - i \omega_k\vert^2}\right\}\geq \frac{1}{\|b\|^2_\ltwoR},
$$
i.e., equivalently,
$$
\inf_{j\in\mathbb N} \left\{ \vert\sigma + i \omega_k\vert, \vert\sigma - i \omega_k\vert \right\}\leq \|b\|_\ltwoR,
$$
which, together with assumption~\eqref{A2}, proves the assertion~3) of Lemma~\ref{lem_spectrum_balls}.
\end{proof}

It turns out that the operator $\cal A$ has only point spectrum, owing to the following lemma.
\begin{lemma}\label{lemma_resolventA}
For each $\lambda\in {\mathbb C}\setminus \sigma_p({\cal A})$, the resolvent $R(\lambda;{\cal A})=({\cal A}-\lambda I)^{-1}$ is a bounded linear operator from $X$ to $X$.
\end{lemma}

\begin{proof}
The lemma is not straightforward and requires some computations.
Exploiting the structure of $\cal A$ in~\eqref{A_op}, we observe that the resolvent operator $R(\lambda;{\cal A})$ maps a vector $\tilde\theta\in X$ to $\theta = ({\cal A}-\lambda I)^{-1}\tilde\theta\in X$ according to the system of equations
\begin{equation}
\begin{aligned}
& \begin{matrix}
-(i\Omega + \lambda I) z + \tilde\phi b = \tilde z,\\
(i\Omega - \lambda I) \zeta - \tilde\phi b = \tilde \zeta,\\
-(i\Omega + \lambda I) p + z = \tilde p,\\
(i\Omega - \lambda I) q + \zeta = \tilde q,
\end{matrix}\quad
\theta = \begin{pmatrix} z \\ \zeta \\ p \\ q
\end{pmatrix},\; \tilde\theta = \begin{pmatrix} \tilde z \\ \tilde \zeta \\ \tilde p \\ \tilde q
\end{pmatrix},\\
& \tilde \phi = \frac{1}{2} b^*(p-q) = \frac{1}{2}\left<p-q,b\right>.
\end{aligned}
\label{resolvent_eq}
\end{equation}
Three cases are possible.

1)~If $\lambda\notin \sigma_p({\cal A}) \cup \{\pm i\omega_1,\pm i \omega_2,...\}$ then, for any $\tilde \theta\in X$, the system~\eqref{resolvent_eq} has a unique solution $\theta\in X$ whose components are
\begin{equation}
\begin{aligned}
& z = (\lambda I + i\Omega)^{-1} \bigl(\tilde \phi(\tilde \theta)  b - \tilde z\bigr),\\
& \zeta = -(\lambda I - i\Omega)^{-1} \bigl(\tilde \phi(\tilde \theta) b + \tilde \zeta\bigr),\\
& p = (\lambda I + i\Omega)^{-2}\bigl(\tilde\phi(\tilde \theta) b - \tilde z\bigr) - (\lambda I + i\Omega)^{-1}\tilde p,\\
& q = - (\lambda I - i\Omega)^{-2}\bigl(\tilde\phi(\tilde \theta) b + \tilde \zeta\bigr) - (\lambda I - i\Omega)^{-1}\tilde q,
\end{aligned}\label{resolvent_comp}
\end{equation}
where the linear functional $\tilde \phi:X\to \mathbb C$ is bounded:
\begin{multline}\label{phi_func}
\tilde \phi (\tilde \theta) = b^*\bigl\{(\lambda I + i\Omega)^{-2}\tilde z - (\lambda I - i\Omega)^{-2}\tilde \zeta \\
+ (\lambda I + i\Omega)^{-1}\tilde p - (\lambda I - i\Omega)^{-1}\tilde q \bigr\}  /\left\{b^*(\lambda I + i\Omega)^{-2}b + b^*(\lambda I - i\Omega)^{-2}b -2\right\}.
\end{multline}
Formulas~\eqref{resolvent_comp} are obtained by solving~\eqref{resolvent_eq} with respect to $\theta$ and treating $\tilde \phi$ as a parameter. Then the value of $\tilde \phi =  \frac{1}{2} b^*(p-q)$ is expressed in terms of $\tilde\theta$ by~\eqref{phi_func}, provided that $\sigma=\lambda$ is not a solution of~\eqref{chareq}.
Altogether, the formulas~\eqref{resolvent_comp}--\eqref{phi_func} define the bounded linear operator $R(\lambda;{\cal A}):X\ni\tilde \theta \mapsto \theta \in X$ for $\lambda\notin \sigma_p({\cal A}) \cup \{\pm i\omega_1,\pm i \omega_2,...\}$.

2)~If $\lambda=i\omega_k$ for some $k\in\mathbb N$, then the coordinates $(z_j,\zeta_j,p_j,q_j)$ of $\theta$
can be inferred from the components of~\eqref{resolvent_comp}
with $j\neq k$ as:
\begin{equation}
\begin{aligned}
& z_j = \frac{b_j \tilde \phi(\tilde \theta)  - \tilde z_j}{\lambda + i\omega_j},\\
& \zeta_j = -\frac{b_j\tilde \phi(\tilde \theta) + \tilde \zeta_j}{\lambda - i\omega_j},\\
& p_j = \frac{b_j\tilde\phi(\tilde \theta) - \tilde z_j}{(\lambda + i\omega_j)^2} - \frac{\tilde p_j}{\lambda + i\omega_j},\\
& q_j = - \frac{b_j\tilde\phi(\tilde \theta) + \tilde \zeta_j}{(\lambda - i\omega_j)^2} - \frac{\tilde q_j}{\lambda - i\omega_j},\quad \text{for}\;j\neq k,
\end{aligned}\label{resolvent_scalar}
\end{equation}
and the coordinates $(z_k,\zeta_k,p_k,q_k)$ of $\theta$ together with $\tilde\phi(\tilde \theta)$ are obtained from~\eqref{resolvent_eq}:
\begin{equation}
\tilde\phi(\tilde \theta) = -\tilde \zeta_k / b_k,\label{phi_k}
\end{equation}
\begin{equation}
\begin{aligned}
& z_k = -\frac{\tilde z_k + \tilde \zeta_k}{2i\omega_k},\\
& \zeta_k = \tilde q_k,\\
& p_k = \frac{\tilde z_k + \tilde \zeta_k}{4 \omega_k^2} - \frac{\tilde p_k}{2i\omega_k},\\
& q_k = \frac{\tilde z_k + \tilde \zeta_k}{4 \omega_k^2} - \frac{\tilde p_k}{2i\omega_k} + \frac{2 \tilde \zeta_k}{b_k^2}+ \frac{1}{b_k}\sum_{j\neq k} b_j(p_j-q_j).
\end{aligned}\label{z_k}
\end{equation}
Hence, for each $\tilde \theta \in X$, the values $\tilde\phi(\tilde \theta)$ and  $(z_j,\zeta_j,p_j,q_j)$ are well-defined by~\eqref{phi_k} and~\eqref{resolvent_scalar} for all $j\neq k$; then the component  $(z_k,\zeta_k,p_k,q_k)$  is computed from~\eqref{z_k}. The above procedure defines the bounded linear operator $R(i\omega_k;{\cal A}):X\ni\tilde \theta \mapsto \theta \in X$.

3)~The case $\lambda=-i\omega_k$ is treated analogously to case~2) with the coordinates $(z_j,\zeta_j,p_j,q_j)$ of $\theta=R(-i\omega_k;{\cal A})\tilde \theta$ defined by~\eqref{resolvent_scalar}, and
$$
\tilde\phi(\tilde \theta) = \tilde z_k / b_k,
$$
$$
\begin{aligned}
& z_k = \tilde p_k,\\
& \zeta_k = \frac{\tilde z_k + \tilde \zeta_k}{2i\omega_k},\\
& p_k = \frac{\tilde z_k + \tilde \zeta_k}{4 \omega_k^2} + \frac{\tilde q_k}{2i\omega_k} + \frac{2 \tilde z_k}{b_k^2}+ \frac{1}{b_k}\sum_{j\neq k} b_j(q_j-p_j),\\
& q_k = \frac{\tilde z_k + \tilde \zeta_k}{4 \omega_k^2} + \frac{\tilde q_k}{2i\omega_k}.
\end{aligned}
$$
The lemma has been proved.
\end{proof}

To estimate the asymptotics of the eigenvalues of $\cal A$, we need the following auxiliary result, which is a consequence of Rouché's theorem~\cite[p.~390]{remmert1991theory}.
\begin{lemma}\label{lemma_Rouche}
Let $f(\lambda)=\lambda^2 - \lambda_0^2  + r(\lambda)$ be a complex-valued function such that $\lambda_0>0$ and
 $r(\lambda)$ is holomorphic in a domain that contains the closed neighborhood
 $$
 {\overline B}_\varepsilon (\lambda_0) = \{ \lambda\in{\mathbb C}\, \vert \, \vert\lambda-\lambda_0\vert \leq \varepsilon\}
 $$
for some $0<\varepsilon<\lambda_0$.
If
\begin{equation}\label{rouche_condition}
\vert r(\lambda_0+\tilde \lambda) \vert < \varepsilon (2 \lambda_0 - \varepsilon)\quad \text{for all}\;\; \tilde \lambda\in{\mathbb C}\; \text{s.t.}\; \vert \tilde \lambda \vert = \varepsilon,
\end{equation}
then $f(\lambda)$ has exactly one zero\footnote{We count each zero of $f(\lambda)$ with its multiplicity.} in $B_\varepsilon(\lambda_0) = \{ \lambda\in{\mathbb C}\, \vert \, \vert\lambda-\lambda_0\vert < \varepsilon\}$.
\end{lemma}

By assuming that the function $r(\lambda)$  is defined in terms of a power series in~\eqref{rouche_condition},
we deduce a useful corollary from Lemma~\ref{lemma_Rouche}.

\begin{corollary}\label{cor_lemma}
Let the power series for $f(\lambda)=\lambda^2-\lambda_0^2 + \sum_{n=3}^\infty c_n \lambda^n$ be convergent for $\vert \lambda \vert <R$, $\lambda_0\in (0,R)$,
and let
\begin{equation}\label{powseries_cond}
\sum_{n=1}^\infty \vert c_{n+2}\vert \lambda_0^n (1+\varkappa)^{n+2} < \varkappa (2-\varkappa),
\end{equation}
for some $\varkappa\in \left(0,\min \{1,\frac{R}{\lambda_0} -1\} \right)$.
Then the equation $f(\lambda)=0$ has exactly one solution in the open $\varepsilon$-neighborhood of $\lambda_0$ with $\varepsilon=\varkappa \lambda_0$.
In particular, for $\varkappa=\frac{1}{2}$ and $ \lambda_0< \frac{2R}{3}$, condition~\eqref{powseries_cond} takes the form
\begin{equation}\label{powseries_simplified}
\sum_{n=1}^\infty \vert c_{n+2}\vert \left( \frac{3 \lambda_0 }{2}\right)^n < \frac{1}{3}.
\end{equation}
\end{corollary}

\begin{lemma}\label{lemma_eigenvalues_asympt}
Under Assumptions \eqref{A1} and \eqref{A2}, the spectrum $\sigma({\cal A})=\{ \sigma_k^+, \sigma_k^-, \sigma_{-k}^+, \sigma_{-k}^- \, \vert \, k\in \mathbb N\}$ of the operator $\cal A$ admits the following representation as $k\to \infty$:
\begin{equation}\label{eigenval_asympt}
\begin{aligned}
& \sigma_k^+ = i\omega_k + \nu_k,\\
& \sigma_k^- = i\omega_k - \overline{\nu_k},\\
& \sigma_{-k}^+ = -i\omega_k + \overline{\nu_k},\\
& \sigma_{-k}^- = -i\omega_k - \nu_k,
\end{aligned}
\end{equation}
where
\begin{equation}\label{lambda0}
\begin{aligned}
&\nu_k = \lambda^0_k  + \varepsilon_k,\\
&\lambda^0_k = \frac{\vert b_k \vert}{\sqrt{2 + \frac{b_k^2}{4\omega_k^2}+2\sum_{j\neq k} \frac{b_j^2 (\omega_k^2 + \omega_j^2)}{(\omega_k^2 - \omega_j^2)^2}}}>0,
\end{aligned}
\end{equation}
with $\varepsilon_k \in \mathbb{C}$ such that $\vert \varepsilon_k\vert < \frac{\lambda^0_k}{2}$.
In particular, ${\rm Re}\, \sigma_{\pm k}^{\pm} \to 0$ as $k\to\infty$.
\end{lemma}
\begin{proof}
Substituting $\sigma = i\omega_k + \lambda$ into~\eqref{chareq}, we get
$$
\begin{aligned}
&\frac{b_k^2}{\lambda^2} + \frac{b_k^2}{(\lambda+2i\omega_k)^2} \\
&+ \sum_{j\neq k}b_j^2\left(\frac{1}{(\lambda+i(\omega_k+\omega_j))^2} + \frac{1}{(\lambda+i(\omega_k-\omega_j))^2} \right) =2.
\end{aligned}
$$
Multiplying this equation by $\lambda^2$ yields
$$
f(\lambda)\equiv \lambda^2 -(\lambda^0_k)^2 + r(\lambda)=0,
$$
where $\lambda^0_k$ is defined by~\eqref{lambda0} and $r(\lambda)= \sum_{n=3}^\infty c_n \lambda^n$,
\begin{equation}
\begin{aligned}
\frac{r(\lambda)}{(\lambda^0_k)^2 \lambda^3} =& - \frac{\lambda+ 4i\omega_k}{4\omega_k^2 (\lambda+2i\omega_k)^2} 
- \sum_{j\neq k}\frac{b_j^2(\lambda+2i(\omega_k+\omega_j))}{b_k^2 (\omega_k+\omega_j)^2 (\lambda+i (\omega_k+\omega_j))^2}\\
& - \sum_{j\neq k}\frac{b_j^2(\lambda+2i(\omega_k-\omega_j))}{b_k^2 (\omega_k-\omega_j)^2 (\lambda+i (\omega_k-\omega_j))^2},\\
c_n=&(n-1)  i^{n-2} (\lambda^0_k)^2 \Bigl\{\frac{1}{(2\omega_k)^n} 
+ \sum_{j\neq k}\frac{b_j^2}{b_k^2} \left(\frac{1}{(\omega_k+\omega_j)^n}+ \frac{1}{(\omega_k-\omega_j)^n} \right)\Bigr\},\; n\geq 3.
\end{aligned}
\label{r_form}
\end{equation}
Under Assumptions~\eqref{A1} and \eqref{A2}, the remainder $r(\lambda)$ is a holomorphic function of $\lambda$ in a disk $\vert \lambda\vert <R$ such that the radius of convergence $R>0$ is the distance from $\lambda=0$ to the nearest singularity of $r(\lambda)$~\cite[Chap.~7, \S 3]{remmert1991theory}, i.e. $R=\min\{R^*,R_*\}$,
\begin{equation}
R^*=\min_{j\neq k}\{\omega_k+\omega_j\},\; R_*=\min_{j\neq k} \{\vert \omega_k - \omega_j \vert \}\geq \omega_*>0.
\label{radius}
\end{equation}

Formula~\eqref{lambda0} implies that
\begin{equation}
0 <\lambda_k^0 < \frac{|b_k|}{\sqrt{2}}.
\label{lambda0est}
\end{equation}
We use this estimate to ensure that the inequality~\eqref{powseries_cond} holds when $\lambda_0 = \lambda_k^0 $, provided that
\begin{equation}
\sum_{n=1}^\infty \vert c_{n+2}\vert \left(\frac{b_k^2(1+\varkappa)}{2}\right)^n < \frac{\varkappa(2-\varkappa)}{(1+\varkappa)^2},\; \varkappa\in (0,1).
\label{series_est1}
\end{equation}
We substitute the coefficients $c_{n+2}$ from~\eqref{r_form} to prove that the inequality~\eqref{series_est1} is satisfied when
\begin{equation}
b_k^2\sum_{n=1}^\infty (n+1) \gamma_{n+2} x^n < \frac{2\varkappa (2-\varkappa)}{(1+\varkappa)^2},\; x= \frac{b_k^2(1+\varkappa)}{2} <1,
\label{series_est2}
\end{equation}
$$
\gamma_n = \left| \frac{1}{(2\omega_k)^n} + \sum_{j\neq k}\frac{b_j^2}{b_k^2} \left(\frac{1}{(\omega_k+\omega_j)^n}+ \frac{1}{(\omega_k-\omega_j)^n} \right) \right|.
$$
We estimate the left-hand side of~\eqref{series_est2} term-by-term by exploiting the triangle inequality, Hölder's inequality, and the identity $\sum_{n=1}^\infty (n+1)\gamma_{n+2} x^n = \frac{d}{dx} \sum_{n=2}^\infty \gamma_{n+1}x^n$:
\begin{equation}
b_k^2\sum_{n=1}^\infty (n+1) \gamma_{n+2} x^n \leq S_1 + S_2 + S_3,
\label{S_formulas}
\end{equation}
$$
\begin{aligned}
S_1 = &\;b_k^2 \sum_{n=1}^\infty \frac{(n+1)x^n}{(2\omega_k)^{n+2}} = b_k^2\frac{d}{dx}\sum_{n=2}^\infty \frac{x^n}{(2\omega_k)^{n+1}} = \frac{b_k^2 x(4\omega_k - x)}{4\omega_k^2 (2\omega_k - x)^2},\\
S_2 = & \,\|b\|^2 \sum_{n=1}^\infty \frac{(n+1)x^n}{(R^*)^{n+2}} = \frac{\|b\|^2 x(2R^* - x)}{(R^*)^2 (R^* - x)^2},\\
S_3 = & \,\|b\|^2 \sum_{n=1}^\infty \frac{(n+1)x^n}{(R_*)^{n+2}} = \frac{\|b\|^2 x(2 R_* - x)}{(R_*)^2 (R_* - x)^2},
\end{aligned}
$$
where $R^*$ and $R_*$ are defined by~\eqref{radius},
% $= \min\{\omega_k-\omega_{k-1},\omega_{k+1}-\omega_k\}$.
and the inequality~\eqref{S_formulas} is strict for $x\neq 0$.
Applying the estimate~\eqref{S_formulas} to~\eqref{series_est2}, we deduce that the condition~\eqref{powseries_cond} is satisfied for some $\varkappa\in (0,1)$ if
\begin{equation}
S_1 + S_2 +S_3 \leq \frac{2\varkappa (2-\varkappa)}{(1+\varkappa)^2},\; x= \frac{b_k^2(1+\varkappa)}{2} <1.
\label{S_condition}
\end{equation}
Since the continuous functions $S_1$, $S_2$, $S_3$ vanish at $x=0$, we conclude that, for any $\varkappa\in(0,1)$,
the conditions~\eqref{S_condition} are satisfied for $x>0$ sufficiently small (i.e., for $k$ large enough) under Assumptions~\eqref{A1} and \eqref{A2}.\

%with the radius of convergence defined by the Cauchy--Hadamard theorem:
%{\footnotesize
%$$
%\frac{1}{R} = \limsup_{n\to \infty} \left| \frac{1}{(2\omega_k)^n} + \sum_{j\neq k}\frac{b_j^2}{b_k^2} %\left(\frac{1}{(\omega_k+\omega_j)^n}+ \frac{1}{(\omega_k-\omega_j)^n} \right) \right|^{1/n}.
%$$}
%such that $\vert r(\lambda) \vert = O(\vert \lambda\vert^3)$ as $\lambda\to 0$.
In particular, for $\varkappa=\frac12$, Corollary~\ref{cor_lemma} implies that there exists a unique root of $f(\lambda)=0$ in the $\frac{\lambda^0_k}{2}$-neighborhood of $\lambda^0_k$. This proves that $\sigma_k^+ = i\omega_k + \lambda^0_k + \varepsilon_k$ is a solution of equation~\eqref{chareq}  with some $\vert \varepsilon_k \vert < \frac{\lambda^0_k}{2}$. The eigenvalues $\sigma_k^-$, $\sigma_{-k}^+$, and $\sigma_{-k}^-$ are obtained from $\sigma_k^+ $ by exploiting the symmetry of $\sigma ({\cal A})$ (Lemma~\ref{lem_spectrum_balls}).
\end{proof}

%\textcolor{red}{
%\begin{remark}\label{rem_lambda_asympt}
%By analyzing the asymptotic behavior of $\lambda_k^0$ in~\eqref{lambda0} under the assumptions of Lemma~\ref{lemma_eigenvalues_asympt}, we observe that
%$$
%\lambda_k^0 \sim \vert b_k\vert \qquad \text{as } k\to+\infty.
%$$
%\end{remark}}

\begin{remark}\label{rem_lambda_asympt}
Under (A1)-(A2), one has
$$
\sum_{j\neq k}\vert b_j\vert^2\frac{\omega_k^2+\omega_j^2}{(\omega_k^2-\omega_j^2)^2}\to 0
\qquad\text{as }k\to+\infty,
$$
hence
$$
\lambda_k^0=\frac{\vert b_k\vert}{\sqrt{2}}\big. \big(1+\mathrm{o}(1)\big),
\qquad k\to+\infty.
$$
In particular, $\lambda_k^0\asymp\vert b_k\vert$, and since $\vert\varepsilon_k\vert\leq\frac{1}{2}\lambda_k^0$ in Lemma~\ref{lemma_eigenvalues_asympt},
one also has $\Re\nu_k\asymp\vert b_k\vert$.
\end{remark}

\medskip

The following lemma states that every generalized eigenvector of $\cal A$ is actually an eigenvector. In less formal terms, this means that ``the Jordan canonical form'' of ${\cal A}$, when considered as an infinite matrix, does not contain any Jordan blocks of size greater than one.
\begin{lemma}\label{statement_A}
Under the assumptions of Lemma~\ref{lemma_eigenvalues_asympt}, let $\hat \sigma\in\mathbb C$ be an eigenvalue of $\cal A$ of the form~\eqref{eigenval_asympt}.
Then the eigenvalue $\hat \sigma$ has algebraic multiplicity $1$.
\end{lemma}
\begin{proof}
The assertion of this lemma follows from the construction of the spectral projection of $\cal A$ corresponding to the isolated spectral point $\hat \sigma$ and the fact that $\hat \sigma$ is a simple root of the characteristic equation~\eqref{chareq}.
\end{proof}

\subsection{Construction of a Riesz basis}\label{sec_Riesz}

In this subsection, we construct a Riesz basis of $X$ consisting of pairwise linear combinations of eigenvectors of $\cal A$.
For a complex number $\sigma\notin \{\pm i\omega_1, \pm i\omega_2, ...\}$, we denote the vector $\theta(\sigma)\in X$ whose components are defined by~\eqref{theta_phi} with $\phi=1$:
\begin{equation}\label{theta_sigma}
\theta(\sigma)=\begin{pmatrix}z \\ \zeta \\ p \\ q\end{pmatrix} = \begin{pmatrix} (\sigma I + i\Omega)^{-1}b \\ - (\sigma I - i\Omega)^{-1}b \\ (\sigma I + i\Omega)^{-2}b \\ -(\sigma I - i\Omega)^{-2}b \end{pmatrix}.
\end{equation}
It follows from the results of Section~\ref{sec_spectrum} that
\begin{equation}
\theta_k^+=\theta(\sigma_k^+),\, \theta_k^-=\theta(\sigma_k^-),\,\theta_{-k}^+ = \theta(\sigma_{-k}^+),\,  \theta_{-k}^- = \theta(\sigma_{-k}^-)
\label{theta_eig}
\end{equation}
are  eigenvectors of $\cal A$,
where $\sigma_k^+$, $\sigma_k^-$, $\sigma_{-k}^+ $,  $\sigma_{-k}^-$ are the eigenvalues given by~\eqref{eigenval_asympt} in Lemma~\ref{lemma_eigenvalues_asympt}. The sequence of eigenvectors of $\cal A$ is complete in $X$, as summarized in the following lemma.

\begin{lemma}\label{statement_B}
Under Assumptions~\eqref{A1} and \eqref{A2}, the closed linear span of all
eigenvectors of $\cal A$ coincides with $X$, i.e.
$$
\overline{\rm span}\left\{\theta_k^+, \theta_k^-, \theta_{-k}^+, \theta_{-k}^- \,\vert \, k\in{\mathbb N}\right\} = X.
$$
\end{lemma}
\begin{proof}
The linear operator ${\cal A}:D({\cal A})\to X$ is densely defined and closed in the Hilbert space $X$,
and its spectrum consists solely of isolated eigenvalues: $\sigma({\cal A})=\sigma_p({\cal A})=\{ \sigma_k^+, \sigma_k^-, \sigma_{-k}^+, \sigma_{-k}^- \, \vert \, k\in \mathbb N\}$.
Consequently, based on Lemmas~2.21 and 2.22 in~\cite{guo2019control}, it follows that ${\rm sp}({\cal A})=X$,
where ${\rm sp}(\cal A)$ is the root subspace of {\cal A} (i.e., the closed linear span of all generalized
eigenvectors of $\cal A$).
Owing to Lemma~\ref{statement_A}, each generalized eigenvector of $\cal A$ is also an eigenvector, thereby proving the assertion of Lemma~\ref{statement_B}.
\end{proof}

Let us consider the ``standard'' orthonormal basis $\{ \hat\theta_k^z, \hat\theta_k^\zeta ,\hat\theta_k^p, \hat\theta_k^q\}_{k=1}^\infty$ in $X$.
By standard, we mean that each vector $\hat\theta_k^z, \hat\theta_k^\zeta ,\hat\theta_k^p, \hat\theta_k^q$ has exactly one coordinate equal to $1$, while the other coordinates are zero.
More precisely, we introduce the linear operators $\Pi_k:X\to {\mathbb C}^4$ which map $\theta$-vectors onto the subspace spanned by the coordinates $(z_k,\theta_k,p_q,q_k)^\top\in {\mathbb C}^4$:
$$
\Pi_k \theta = \begin{pmatrix}z_k \\ \zeta_k \\ p_k \\ q_k\end{pmatrix},\quad k\in{\mathbb N}.
$$
Then, for the standard orthonormal basis, we have
$$
\Pi_k \hat\theta_k^z= \begin{pmatrix}1\\0\\0\\0\end{pmatrix},
\Pi_k \hat\theta_k^\zeta= \begin{pmatrix}0\\1\\0\\0\end{pmatrix},
\Pi_k \hat\theta_k^p= \begin{pmatrix}0\\0\\1\\0\end{pmatrix},
\Pi_k \hat\theta_k^q= \begin{pmatrix}0\\0\\0\\1\end{pmatrix},%\quad k\in{\mathbb N},
$$
and $\Pi_n \hat\theta_k^z=\Pi_n \hat\theta_k^\zeta=\Pi_n \hat\theta_k^p=\Pi_n \hat\theta_k^q=0\in{\mathbb C}^4$ for all $n\neq k$.

\begin{lemma}\label{lem_quadratic}
Under Assumptions~\eqref{A1}, \eqref{A2}, and~\eqref{A4},
let the vectors  $\{\theta_k^z, \theta_k^\zeta, \theta_k^p, \theta_k^q \}_{k=1}^\infty$ be defined by
\begin{equation}\label{Riesz_basis}
\begin{aligned}
& \theta_k^z = \frac{1}{2 b_k {\rm Re}\, \nu_k} \left( {\overline\nu_k}^2\theta_{-k}^+ - {\nu_k}^2\theta_{-k}^- \right), \\
& \theta_k^\zeta = -\frac{1}{2 b_k {\rm Re}\, \nu_k} \left({\nu_k}^2  \theta_{k}^+ - {\overline\nu_k}^2 \theta_{k}^-   \right), \\
& \theta_k^p = \frac{1}{2 b_k} \left( {\overline\nu_k}^2\theta_{-k}^+ + {\nu_k}^2 \theta_{-k}^- \right), \\
& \theta_k^q = -\frac{1}{2 b_k} \left( {\nu_k}^2\theta_{k}^+ + {\overline\nu_k}^2 \theta_{k}^- \right),\;k\in{\mathbb N},
\end{aligned}
\end{equation}
where the $\nu_k$ are defined by~\eqref{lambda0} in Lemma~\ref{lemma_eigenvalues_asympt}.

Then the sequence $\{\theta_k^z, \theta_k^\zeta, \theta_k^p, \theta_k^q \}_{k=1}^\infty$ is quadratically close to the standard orthonormal basis $\{ \hat\theta_k^z, \hat\theta_k^\zeta ,\hat\theta_k^p, \hat\theta_k^q\}_{k=1}^\infty$ in $X$, i.e.
\begin{equation}
\sum_{k=1}^\infty \bigl(  \|\theta_k^z-\hat\theta_k^z\|^2 + \|\theta_k^\zeta-\hat\theta_k^\zeta\|^2 
+\|\theta_k^p-\hat\theta_k^p\|^2 +\|\theta_k^q-\hat\theta_k^q\|^2  \bigr) <\infty.
\label{quadclose}
\end{equation}
\end{lemma}

\begin{proof}
We rewrite the above sum of squares using the operators $\Pi_k$:
\begin{equation}
\sum_{k=1}^\infty \|\theta_k^s-\hat\theta_k^s\|^2 = \sum_{k=1}^\infty \sum_{j=1}^\infty \|\Pi_j (\theta_k^s-\hat\theta_k^s)\|^2
= \sum_{k=1}^\infty  \|\Pi_k (\theta_k^s-\hat\theta_k^s)\|^2 + \sum_{k=1}^\infty \sum_{j\neq k} \|\Pi_j \theta_k^s\|^2,
\label{norm2_identity}
\end{equation}
where the symbol ``$s$'' represents any of the superscripts $z$, $\zeta$, $p$, $q$,
and the norm in $\mathbb C^4$ is defined as Hermitian.

To show that the two sequences are quadratically close~\cite{gohberg1978introduction}, we first observe that
$$
\Pi_k \theta_k^z = \begin{pmatrix} 1 \\  - \frac{(\Renu)^2+(4\omega_k + \Imnu)\Imnu}{(\Renu)^2 + (2\omega_k + \Imnu)^2} \\ 0 \\ -4i\omega_k \frac{2\omega_k\Imnu+\vert \nu_k \vert^2}{[(\Renu)^2 + (2\omega_k + \Imnu)^2]^2}  \end{pmatrix},
$$
$$
\Pi_k \theta_k^\zeta = \begin{pmatrix} - \frac{(\Renu)^2+(4\omega_k + \Imnu)\Imnu}{(\Renu)^2 + (2\omega_k + \Imnu)^2} \\ 1 \\ 4i\omega_k \frac{2\omega_k\Imnu+\vert \nu_k \vert^2}{[(\Renu)^2 + (2\omega_k + \Imnu)^2]^2} \\ 0\end{pmatrix},
$$
$$
\Pi_k \theta_k^p = \begin{pmatrix}  -i \Imnu \\ -i\frac{2\omega_k[(\Renu)^2-(\Imnu)^2] - \vert \nu_k\vert^2 \Imnu}{(\Renu)^2 + (2\omega_k + \Imnu)^2} \\ 1 \\ \frac{(2\omega_k (\Renu+\Imnu)+\vert \nu_k\vert^2)(2\omega_k (\Renu-\Imnu)-\vert \nu_k\vert^2)}{[(\Renu)^2 + (2\omega_k + \Imnu)^2]^2} \end{pmatrix},
$$
$$
\Pi_k \theta_k^q = \begin{pmatrix} i\frac{2\omega_k[(\Renu)^2-(\Imnu)^2] - \vert \nu_k\vert^2 \Imnu}{(\Renu)^2 + (2\omega_k + \Imnu)^2} \\ i\Imnu \\ \frac{(2\omega_k (\Renu+\Imnu)+\vert \nu_k\vert^2)(2\omega_k (\Renu-\Imnu)-\vert \nu_k\vert^2)}{[(\Renu)^2 + (2\omega_k + \Imnu)^2]^2} \\ 1\end{pmatrix}.
$$
The estimate~\eqref{lambda0est} of Lemma~\ref{lemma_eigenvalues_asympt} implies that
\begin{equation}
\vert \nu_k\vert < \sqrt{2} \vert b_k\vert.
\label{nu_est}
\end{equation}
Then, from inequality~\eqref{nu_est} and Assumptions~\eqref{A1} and \eqref{A2}, it follows that
\begin{equation}\label{k_sum}
\sum_{k=1}^\infty  \|\Pi_k (\theta_k^s-\hat\theta_k^s)\|^2 < \infty
\end{equation}
for each ``$s$'' representing the components $z$, $\zeta$, $p$, $q$.

Due to~\eqref{norm2_identity}, it remains to prove that
\begin{equation}
\sum_{k=1}^\infty \sum_{j\neq k} \|\Pi_j \theta_k^s\|^2 <\infty.
\label{sum_kj}
\end{equation}
Straightforward computations show that
$$
\Pi_j \theta_k^z = \frac{b_j}{2 b_k {\rm Re}\, \nu_k} \begin{pmatrix}
\frac{{\overline\nu_k}^2}{i(\omega_j-\omega_k)+\overline\nu_k} + \frac{\nu_k^2}{i(\omega_k-\omega_j)+\nu_k}
 \\
 \frac{{\overline\nu_k}^2}{i(\omega_j+\omega_k)-\overline\nu_k} - \frac{\nu_k^2}{i(\omega_j+\omega_k)+\nu_k}
 \\
 \frac{{\overline\nu_k}^2}{(i(\omega_j-\omega_k)+\overline\nu_k)^2} - \frac{\nu_k^2}{(i(\omega_k-\omega_j)+\nu_k)^2}
 \\
 -\frac{{\overline\nu_k}^2}{(i(\omega_j+\omega_k)-\overline\nu_k)^2} + \frac{\nu_k^2}{(i(\omega_j+\omega_k)+\nu_k)^2}
 \end{pmatrix},
$$
$$
\Pi_j \theta_k^\zeta = \frac{b_j}{2 b_k {\rm Re}\, \nu_k}\begin{pmatrix}
-\frac{\nu_k^2}{i(\omega_k+\omega_j)+\nu_k} + \frac{\overline\nu_k^2}{i(\omega_k+\omega_j)-\overline\nu_k}
 \\
 \frac{\nu_k^2}{i(\omega_k-\omega_j)+\nu_k} - \frac{\overline\nu_k^2}{i(\omega_k-\omega_j)-\overline\nu_k}
 \\
-\frac{\nu_k^2}{(i(\omega_k+\omega_j)+\nu_k)^2} + \frac{\overline\nu_k^2}{(i(\omega_k+\omega_j)-\overline\nu_k)^2}
\\
\frac{\nu_k^2}{(i(\omega_k-\omega_j)+\nu_k)^2} - \frac{\overline\nu_k^2}{(i(\omega_k-\omega_j)-\overline\nu_k)^2}
 \end{pmatrix},
$$
$$
\Pi_j \theta_k^p = \frac{b_j}{2 b_k}\begin{pmatrix}
\frac{{\overline\nu_k}^2}{i(\omega_j-\omega_k)+\overline\nu_k} + \frac{\nu_k^2}{i(\omega_j-\omega_k)-\nu_k}
\\
\frac{{\overline\nu_k}^2}{i(\omega_j+\omega_k)-\overline\nu_k} + \frac{\nu_k^2}{i(\omega_j+\omega_k)+\nu_k}
\\
\frac{{\overline\nu_k}^2}{(i(\omega_j-\omega_k)+\overline\nu_k)^2} + \frac{\nu_k^2}{(i(\omega_j-\omega_k)-\nu_k)^2}
\\
-\frac{{\overline\nu_k}^2}{(i(\omega_j+\omega_k)-\overline\nu_k)^2} - \frac{\nu_k^2}{(i(\omega_j+\omega_k)+\nu_k)^2}
 \end{pmatrix},
$$
$$
\Pi_j \theta_k^q = \frac{b_j}{2 b_k}\begin{pmatrix}
-\frac{\nu_k^2}{i(\omega_k+\omega_j)+\nu_k} - \frac{\overline\nu_k^2}{i(\omega_k+\omega_j)-\overline\nu_k}
 \\
 \frac{\nu_k^2}{i(\omega_k-\omega_j)+\nu_k} + \frac{\overline\nu_k^2}{i(\omega_k-\omega_j)-\overline\nu_k}
 \\
 -\frac{\nu_k^2}{(i(\omega_k+\omega_j)+\nu_k)^2} - \frac{\overline\nu_k^2}{(i(\omega_k+\omega_j)-\overline\nu_k)^2}
 \\
 \frac{\nu_k^2}{(i(\omega_k-\omega_j)+\nu_k)^2} + \frac{\overline\nu_k^2}{(i(\omega_k-\omega_j)-\overline\nu_k)^2}
 \end{pmatrix},\; j\neq k.
$$
As in the proof of Lemma~\ref{lemma_eigenvalues_asympt}, we consider a $\varkappa\in (0,1)$ such that
the conditions of Corollary~\ref{cor_lemma} are satisfied for $\lambda_0=\lambda^0_k$. Then, due to~\eqref{nu_est}, we obtain
\begin{equation}\label{kappa_est}
\left|\frac{\nu_k^2}{b_k \Renu}\right| < \frac{\sqrt{2}(1+\varkappa)}{1-\varkappa}.
\end{equation}
Estimating the terms in the sum~\eqref{sum_kj} and applying inequalities~\eqref{nu_est} and~\eqref{kappa_est}, along with Assumptions~\eqref{A1}, \eqref{A2} and~\eqref{A4}, we conclude that the condition~\eqref{sum_kj} holds for each ``$s$'' representing the components $z$, $\zeta$, $p$, and $q$.
Then~\eqref{k_sum} and~\eqref{sum_kj} together with~\eqref{norm2_identity} imply the assertion of Lemma~\ref{lem_quadratic}.
\end{proof}

As the system of vectors~\eqref{Riesz_basis} and the eigenvectors of $\cal A$ are related through a bounded and invertible linear transformation, their relationship is summarized in the following lemma.

\begin{lemma}\label{statement_CD}
Under Assumptions~\eqref{A1} and \eqref{A2}, let the vectors $\{ \theta_k^z, \theta_k^\zeta ,\theta_k^p, \theta_k^q \, \vert \, k\in {\mathbb N}\}$ be defined by~\eqref{Riesz_basis}.
Then:
\begin{enumerate}
\item[i)]the sequence $\{ \theta_k^z, \theta_k^\zeta ,\theta_k^p, \theta_k^q\}_{k=1}^\infty$ is $\omega$-linearly independent if and only if the sequence $\left\{\theta_k^+, \theta_k^-, \theta_{-k}^+, \theta_{-k}^- \right\}_{k=1}^\infty$ is $\omega$-linearly independent;
\item[ii)]$
\overline{\rm span}\left\{\theta_k^+, \theta_k^-, \theta_{-k}^+, \theta_{-k}^- \right\}_{k=1}^\infty =\overline{\rm span}\{ \theta_k^z, \theta_k^\zeta ,\theta_k^p, \theta_k^q \}_{k=1}^\infty
$.
\end{enumerate}
\end{lemma}

\begin{proof}
Assume \eqref{A1} (so that $b_k\neq 0$ for all $k$) and assume that the spectrum is hyperbolic, that is, $\sigma_p(A)\cap i\mathbb{R}=\emptyset$.
Then, for every $k\in\mathbb{N}$, relations~\eqref{Riesz_basis} define an invertible linear change of generators between
$\mathrm{span}\{\theta_{-k}^+,\theta_{-k}^-\}$ and $\mathrm{span}\{\theta_k^z,\theta_k^p\}$,
and between
$\mathrm{span}\{\theta_{k}^+,\theta_{k}^-\}$ and $\mathrm{span}\{\theta_k^\zeta,\theta_k^q\}$.
More precisely, one has the explicit inverse formulas
$$
\theta_{-k}^+=\frac{b_k}{\overline{\nu_k}^{\,2}}\Big(\theta_k^p+\Re(\nu_k)\,\theta_k^z\Big),
\;
\theta_{-k}^-=\frac{b_k}{\nu_k^{2}}\Big(\theta_k^p-\Re(\nu_k)\,\theta_k^z\Big),
$$
and
$$
\theta_{k}^+=-\frac{b_k}{\overline{\nu_k}^{\,2}}\Big(\theta_k^q+\Re(\nu_k)\,\theta_k^\zeta\Big),
\;
\theta_{k}^-=-\frac{b_k}{\nu_k^{2}}\Big(\theta_k^q-\Re(\nu_k)\,\theta_k^\zeta\Big).
$$
From~\eqref{Riesz_basis}, for fixed $k$,
$$
2b_k\Re(\nu_k)\,\theta_k^z
= \overline{\nu_k}^{\,2}\theta_{-k}^+ - \nu_k^2\theta_{-k}^-,
\;
2b_k\,\theta_k^p
= \overline{\nu_k}^{\,2}\theta_{-k}^+ + \nu_k^2\theta_{-k}^-.
$$
Adding and subtracting these identities yields
$$
2b_k\big(\theta_k^p+\Re(\nu_k)\theta_k^z\big)
=
2\overline{\nu_k}^{\,2}\theta_{-k}^+,
$$
$$
2b_k\big(\theta_k^p-\Re(\nu_k)\theta_k^z\big) = 2\nu_k^{2}\theta_{-k}^-,
$$
which gives the first pair of formulas. The second pair is obtained analogously from the relations defining $\theta_k^\zeta$ and $\theta_k^q$ in~\eqref{Riesz_basis}.

It remains to justify that the denominators do not vanish. By~\eqref{A1}, $b_k\neq 0$. By hyperbolicity, there are no purely imaginary eigenvalues, hence $\Re(\nu_k)\neq 0$ (otherwise $\sigma_k^\pm=i\omega_k\pm\nu_k$ would be purely imaginary). Therefore $\nu_k\neq 0$, which proves invertibility.
\end{proof}

\begin{proposition}\label{thm_Riesz}
Under Assumptions~\eqref{A1}, \eqref{A2}, and~\eqref{A4}, the sequence $\{ \theta_k^z, \theta_k^\zeta ,\theta_k^p, \theta_k^q \}_{k=1}^\infty$ is a Riesz basis of $X$, consisting of linear combinations of eigenfunctions of $\cal A$.
\end{proposition}
\begin{proof}
Due to Lemma~\ref{lem_quadratic}, the sequence $\{ \theta_k^z, \theta_k^\zeta ,\theta_k^p, \theta_k^q \}_{k=1}^\infty$ is quadratically close to the standard {orthonormal} basis in~$X$.
 Therefore, according to Bari's theorem~\cite{gohberg1978introduction},
the sequence $\{ \theta_k^z, \theta_k^\zeta ,\theta_k^p, \theta_k^q\}_{k=1}^\infty$ is a Riesz basis in $X$, provided that it is $\omega$-linearly independent.
Thus, at the light of part~i) of Lemma~\ref{statement_CD}, it suffices to prove that
\begin{equation}\label{omega_indep}
\left\{\theta_k^+, \theta_k^-, \theta_{-k}^+, \theta_{-k}^- \right\}_{k=1}^\infty\quad\text{is}\;\; \omega\text{-linearly independent}.
\end{equation}
Property~\eqref{omega_indep} is proved by adapting the approach used in the proof of~\cite[Theorem~2.38]{guo2019control} to the case of linear combinations of eigenvectors, taking into account Lemma~\ref{statement_CD}.
\end{proof}

{
For a given $\theta\in X$, we consider its expansion with respect to the Riesz basis in $X$, defined in Proposition~\ref{thm_Riesz}:
$$
\theta = \sum_{k=1}^\infty \left( a_k^z \theta_k^z + a_k^\zeta \theta_k^\zeta + a_k^p \theta_k^p + a_k^q \theta_k^q \right).
$$
This expansion is related to the bounded linear operator $\pi_0:X\to X$ such that
\begin{equation}\label{pi0t}
X\ni a= \begin{pmatrix}a^z \\ a^\zeta \\ a^p \\ a^q \end{pmatrix} \mapsto \pi_0 a =\theta= \begin{pmatrix} z \\ \zeta \\ p \\ q \end{pmatrix}\in X.
\end{equation}
According to the formulas~\eqref{Riesz_basis}, the operator $\pi_0:X\to X$ acts as
\begin{equation}\label{pi0}
\pi_0 a = \pi_z a^z + \pi_\zeta a^\zeta + \pi_p a^p + \pi_q a^q,
\end{equation}
where the matrix representations of the operators $\pi_z,\pi_\zeta,\pi_p,\pi_q: \ltwoC\to X$ are composed of the columns from~\eqref{Riesz_basis}:
\begin{equation}\label{pi_components}
\begin{aligned}
\pi_z &= (\theta_1^z, \theta_2^z,... ),\;
\pi_\zeta = (\theta_1^\zeta, \theta_2^\zeta,... ),\\
\pi_p &= (\theta_1^p, \theta_2^p,... ),\;
\pi_q = (\theta_1^q, \theta_2^q,... ).
\end{aligned}
\end{equation}
Here, we follow the convention that $\pi_s a^s = \sum_{k=1}^\infty \theta_k^s a^s_k$, for each ``$s$'' representing the components $z$, $\zeta$, $p$, $q$. Note that both $\pi_0$ and $\pi_0^{-1}$ are bounded linear operators from $X$ to $X$, because $\{\theta_k^z, \theta_k^\zeta, \theta_k^p, \theta_k^q \}_{k=1}^\infty$ is a Riesz basis in $X$.
}

{
\paragraph{Isomorphism property of $\pi_0$ on $V_{0,\gamma}^X$}
For $\gamma\in\mathbb R$, define
$$
V_{0,\gamma}^X
=\Bigl\{\theta=(z,\zeta,p,q)\in X \ | \ 
\norm{\theta}_{V_{0,\gamma}^X}^2
=\sum_{k\geq 1}\frac{|z_k|^2+|\zeta_k|^2+|p_k|^2+|q_k|^2}{|b_k|^{2\gamma}}<\infty\Bigr\}.
$$
Let $W_\gamma:V_{0,\gamma}^X\to X$ be the diagonal isometry, i.e. $(W_\gamma\theta)_k=\theta_k/|b_k|^\gamma$.}

%Fix $\rho>0$ (so $\gamma=-\rho<0$).

{
\begin{lemma}\label{lem:pi0-weighted}
Let \eqref{A1}, \eqref{A2},~\eqref{A4}, and~\eqref{A5} hold
for some $\rho>0$.
Then the operator $\pi_0$, defined by~\eqref{pi0}, extends to a bounded isomorphism
$$
\pi_0:\ V_{0,-\rho}^X\longrightarrow V_{0,-\rho}^X,
\qquad
\pi_0^{-1}:\ V_{0,-\rho}^X\longrightarrow V_{0,-\rho}^X.
$$
\end{lemma}}

\begin{proof}
{Consider the conjugated operator on $X$:
$$
\widetilde\pi_{0,\rho}=W_{-\rho}\,\pi_0\,W_{-\rho}^{-1}.
$$
Since $W_{-\rho}$ is an isometric isomorphism $V_{0,-\rho}^X\to X$, bounded invertibility of $\pi_0$ on $V_{0,-\rho}^X$ is equivalent to bounded invertibility of $\widetilde\pi_{0,\rho}$ on $X$.
Define the renormalized family in $X$ by
$$
\widetilde\theta_{k,\rho}^s = W_{-\rho}\big(|b_k|^{-\rho}\,\theta_k^s\big),
\qquad s\in\{z,\zeta,p,q\}.
$$
A direct coefficient computation shows that $\widetilde\pi_{0,\rho}$ is the ``synthesis" operator of
$\{\widetilde\theta_{k,\rho}^z,\widetilde\theta_{k,\rho}^\zeta,\widetilde\theta_{k,\rho}^p,\widetilde\theta_{k,\rho}^q\}_{k\geq 1}$.
Thus, by Bari's theorem, it suffices to prove that this renormalized family is quadratically close to the standard orthonormal basis of $X$ and $\omega$-linearly independent.}

\smallskip
\noindent\emph{Diagonal part.}
For each $s\in\{z,\zeta,p,q\}$ and each $k$, one has $\Pi_k\widetilde\theta_{k,\rho}^s=\Pi_k\theta_k^s$ because the weight factors cancel at $j=k$.
Hence the diagonal estimates from Lemma~\ref{lem_quadratic} remain unchanged and yield
$$
\sum_{k\geq 1}\Vert\Pi_k(\widetilde\theta_{k,\rho}^s-\widehat\theta_k^s)\Vert^2<\infty.
$$

\smallskip
\noindent\emph{Off-diagonal part.}
For $j\neq k$,
$$
\Pi_j\widetilde\theta_{k,\rho}^s
=|b_j|^{\rho}|b_k|^{-\rho}\,\Pi_j\theta_k^s.
$$
The proof of Lemma~\ref{lem_quadratic} yields a uniform estimate (for each $s$) of the form
$$
\Vert\Pi_j\theta_k^s\Vert\leq C \frac{|b_j|}{|\omega_j-\omega_k|},\qquad j\neq k,
$$
with a constant $C$ independent of $j,k$ and of $s\in\{z,\zeta,p,q\}$,
hence
$$
\Vert\Pi_j\widetilde\theta_{k,\rho}^s\Vert^2
\leq C^2 \frac{|b_j|^{2(1+\rho)}}{|b_k|^{2\rho}} \frac{1}{(\omega_j-\omega_k)^2}.
$$
Summing in $j\neq k$ and then in $k$ yields quadratic closeness by \eqref{A5}.

\smallskip
\noindent\emph{$\omega$-linear independence.}
Multiplying each vector $\theta_k^s$ by a nonzero scalar preserves $\omega$-linear independence. Thus Bari's theorem applies and yields that the renormalized family is a Riesz basis of $X$.
Therefore $\widetilde\pi_{0,\rho}$ is bounded and invertible on $X$, and conjugating back concludes.
\end{proof}

\subsection{Polynomial decay conditions}
Our study of the large-time behavior of solutions to the considered control system is based on the following lemma.

\begin{lemma}\label{lem_polynomial}
Let the sequence $\{\sigma_k^-\}_{k\in\mathbb N}$ be defined by~\eqref{eigenval_asympt} and~\eqref{lambda0}, where $\omega_k$ and $b_k$ satisfy Assumptions~\eqref{A1}, {\eqref{A2}, and~\eqref{A3}.
Then, for any $\beta>0$, there exists a constant $C_\beta>0$ such that
\begin{equation}\label{exp_poly_decay}
|e^{\sigma_k^- t}| \leq C_\beta |b_k|^{-\beta} (t+1)^{-\beta}\quad \text{for all}\;\; t\geq 0,\; k\in{\mathbb N}.
\end{equation}
Here, $C_\beta$ depends on $\beta$, $b$, and the constants in Lemma~\ref{lemma_eigenvalues_asympt}, but not on $t$ or $k$.}
\end{lemma}
\begin{proof}
{
We start with the following scalar inequality: for every $\beta>0$,
\begin{equation}\label{exp_ineq}
e^{-y}\leq \beta^\beta e^{-\beta} y^{-\beta}\quad \text{for all}\; y>0.
\end{equation}
Indeed, the maximum of $y^\beta e^{-y}$ over $y>0$ is attained at $y=\beta$ and equals $\beta^\beta e^{-\beta}$.
Now, we apply~\eqref{exp_ineq} to $e^{\sigma_k^- t}$.
Since
$\sigma_k^-=i\omega_k-\overline{\nu_k}$, we have
$$
|e^{\sigma_k^- t}|=e^{-(\Renu)t}=e^{\Renu}e^{-(\Renu)(t+1)}.
$$
Fix an arbitrary $\beta>0$ and apply inequality~\eqref{exp_ineq} to the above representation with $y=(\Renu)(t+1)$.
As a result, we obtain
\begin{equation}\label{diag_decay}
|e^{\sigma_k^- t}|\leq e^{-\beta+\Renu} \beta^\beta (\Renu)^{-\beta}(t+1)^{-\beta}\; \text{for all}\; t\geq 0, k\in\mathbb N.
\end{equation}
Under Assumptions \eqref{A1}, \eqref{A2}, and~\eqref{A3}, Lemma~\ref{lemma_eigenvalues_asympt} implies $\Renu \asymp |b_k|$ (for large $k$;  the finitely many remaining $k$ can be absorbed into the constant).
Hence, we obtain the general estimate~\eqref{exp_poly_decay} from~\eqref{diag_decay}.}
\end{proof}

\begin{remark}\label{rem_backward}
Because of the representations~\eqref{eigenval_asympt}, the estimate~\eqref{exp_poly_decay} of Lemma~\ref{lem_polynomial} can also be rewritten in backward time as follows:
\begin{equation}\label{exp_poly_decay_backward}
|e^{\sigma_k^+ t}| \leq C_\beta |b_k|^{-\beta} (T-t+1)^{-\beta}\quad \text{for all}\;\; t\leq T,\; k\in{\mathbb N}.
\end{equation}
\end{remark}

%By extending the above construction to the system $\dot\upsilon(t) = {\cal A}_\upsilon \upsilon(t)$ in backward time with $\upsilon(t)\in \ltwoC$ and ${\cal A}_\upsilon = i\Omega+{\cal V}$,
%\begin{equation}\label{V_op}
%{\cal V}:\ltwoC\to\ltwoC,\; {\cal V} = {\rm diag}(\nu_1, \nu_2,...),
%\end{equation}
%we obtain the following corollary of Lemma~\ref{lem_polynomial}.

%\begin{corollary}\label{cor_polynomial}
%Let the sequence $\{\sigma_k^+\}_{k\in\mathbb N}$ be defined by~\eqref{eigenval_asympt} and~\eqref{lambda0}, where $\omega_k$ and $b_k$ satisfy Assumptions~\eqref{A1}, \eqref{A2}, \eqref{A4}, and \eqref{A5} for some $\alpha>0$.
%Then, there exists a constant $K>0$ such that
%\begin{equation}\label{exp_poly_decay}
%|e^{\sigma_k^+ (t-T)}| \leq K |\sigma_k^+| (T-t+1)^{-1/\alpha},
%\end{equation}
%for all $T\in \mathbb R$, $t\leq T$, and $k\in{\mathbb N}$.
%\end{corollary}

%\begin{remark}
%It can be shown that the condition
%\begin{equation}\label{alpha_log}
%\alpha > \limsup_{n\to \infty}\left( - \frac{\log \vert b_n\vert }{\log \omega_n} \right)
%\end{equation}
%implies~\eqref{A5}.
%\end{remark}

A key result concerning the two-sided polynomial asymptotics of solutions to the variational system is given in the following proposition.

\begin{proposition}\label{thm_turnpike}
{
Let Assumptions \eqref{A1}, \eqref{A2},~\eqref{A3}, and~\eqref{A5} be satisfied
for some constant $\rho>1$, and set $\beta=\rho-1$, $\gamma=-\rho$ (so that, by Lemma~\ref{lem:pi0-weighted}, $\pi_0$ is a bounded isomorphism on $V_{0,\gamma}^X$).
Then there exists a constant $C>0$ such that, for every $T>0$ and every solution $\Delta(\cdot)\in C([0,T];X)$ of~\eqref{delta_op},
\begin{equation}\label{turnpike}
\norm{\Delta(t)}_{V_{0,\gamma}^X}
\leq C\left(
\frac{\norm{\Delta(0)}_X}{(t+1)^\beta}
+\frac{\norm{\Delta(T)}_X}{(T-t+1)^\beta}
\right),\quad t\in[0,T].
\end{equation}}
\end{proposition}
\begin{proof}
{
For a solution $\Delta(\cdot)\in C([0,T];X)$ of~\eqref{delta_op}, consider the corresponding solution $\theta(\cdot)\in C([0,T];X)$ of~\eqref{theta_c}.  Then the function $a(t)=\pi_0^{-1}\theta(t)$, where $\pi_0:X\to X$ is defined in~\eqref{pi0t}}, satisfies the following abstract differential equation:
\begin{equation}\label{aR_abstract}
\dot a(t) = {\cal A}_R a(t),\quad a(t)\in X,
\end{equation}
where
$$
\begin{aligned}
D({\cal A}_R) = \{a\in X\,|\,&\sum_{k=1}^\infty \omega_k^2 \left( |a_k^z|^2+|a_k^\zeta|^2+|a_k^p|^2+|a_k^q|^2 \right)\\
&<\infty\}.
\end{aligned}
$$
In coordinates, \eqref{aR_abstract} takes the form:
\begin{equation}\label{a_zp}
\begin{aligned}
&\frac{d}{dt}\begin{pmatrix}a_k^z \\ a_k^p\end{pmatrix} = A_k^{z p} \begin{pmatrix}a_k^z \\ a_k^p\end{pmatrix},\;
\frac{d}{dt}\begin{pmatrix}a_k^\zeta \\ a_k^q\end{pmatrix} = A_k^{\zeta q} \begin{pmatrix}a_k^\zeta \\ a_k^q\end{pmatrix},\; k\in {\mathbb N},
\end{aligned}
\end{equation}
\begin{equation}\label{AMN}
\begin{aligned}
&A_k^{z p} = \left( M_k^\top \right)^{-1} \Sigma_{-k} M_k^\top,\; A_k^{\zeta q} = \left( N_k^\top \right)^{-1} \Sigma_{k} N_k^\top,\\
&\Sigma_{-k} = \begin{pmatrix}
\sigma_{-k}^+ & 0 \\
0 & \sigma_{-k}^-
\end{pmatrix},\;
\Sigma_{k} = \begin{pmatrix}
\sigma_{k}^+ & 0 \\
0 & \sigma_{k}^-
\end{pmatrix},\\
&M_k = \frac{1}{2b_k}\begin{pmatrix}
\frac{\overline{\nu_k}^2}{\Renu} & - \frac{\nu_k^2}{\Renu} \\
\overline{\nu_k}^2 & {\nu_k}^2
\end{pmatrix}, \qquad
N_k = -\frac{1}{2b_k}\begin{pmatrix}
\frac{\nu_k^2}{\Renu} & - \frac{\overline{\nu_k}^2}{\Renu} \\
\nu_k^2 & \overline{\nu_k}^2
\end{pmatrix}.
\end{aligned}
\end{equation}

By introducing the change of variables
\begin{equation}\label{chi_upsilon_a}
\begin{pmatrix}\upsilon_k \\ \chi_k
\end{pmatrix} =
N_k^\top
\begin{pmatrix}a_k^\zeta \\ a_k^q
\end{pmatrix},\;
\begin{pmatrix}\tilde\upsilon_k \\ \tilde\chi_k
\end{pmatrix} =
M_k^\top
\begin{pmatrix}a_k^z \\ a_k^p
\end{pmatrix},
\end{equation}
the system~\eqref{a_zp} is written as
\begin{equation}\label{chi_upsilon_k}
\begin{aligned}
&\dot \chi_k(t) = \sigma_k^- \chi_k(t),\;
\dot \upsilon_k(t) = \sigma_k^+ \upsilon_k (t),\\
&\dot {\tilde\chi}_k(t) = \sigma_{-k}^- \tilde\chi_k(t),\;
\dot {\tilde\upsilon}_k(t) = \sigma_{-k}^+ \tilde\upsilon_k (t),
\quad k\in{\mathbb N},
\end{aligned}
\end{equation}
where the $\sigma_{\pm k}^\pm$ are defined in~\eqref{eigenval_asympt}.

The abstract formulation of the system~\eqref{chi_upsilon_k} leads to
\begin{equation}\label{chi_upsilon}
\begin{aligned}
&\dot \chi(t) = {\cal A}_\chi \chi(t),\\
&\dot \upsilon(t)= {\cal A}_\upsilon \upsilon(t),\\
&\dot {\tilde\chi}(t) = \tilde{\cal A}_\chi \tilde\chi(t),\\
&\dot {\tilde\upsilon}(t)= \tilde{\cal A}_\upsilon \tilde\upsilon(t),\; \chi(t), \upsilon(t), \tilde\chi(t), \tilde\upsilon(t) \in \ltwoC,
\end{aligned}
\end{equation}
or, equivalently,
\begin{equation}\label{chi_upsilon_tau}
\dot \tau(t) = {\cal A}_\tau \tau(t),\quad \tau(t)\in X = (\ltwoC)^4.
\end{equation}
Here, the linear operators ${\cal A}_\chi:D({\cal A}_\chi)\to \ltwoC$, ${\cal A}_\upsilon:D({\cal A}_\upsilon)\to \ltwoC$,  $\tilde{\cal A}_\chi:D(\tilde{\cal A}_\chi)\to \ltwoC$, $\tilde{\cal A}_\upsilon:D(\tilde{\cal A}_\upsilon)\to \ltwoC$ are represented by their diagonal matrices:
$$
\begin{aligned}
{\cal A}_\chi &= {\rm diag}(\sigma_1^-,\sigma_2^-,...),\; {\cal A}_\upsilon = {\rm diag}(\sigma_1^+,\sigma_2^+,...),\\
\tilde{\cal A}_\chi &= {\rm diag}(\sigma_{-1}^-,\sigma_{-2}^-,...),\; \tilde {\cal A}_\upsilon = {\rm diag}(\sigma_{-1}^+,\sigma_{-2}^+,...),\\
D({\cal A}_\chi) &= D({\cal A}_\upsilon) = D(\tilde{\cal A}_\chi) 
= D(\tilde{\cal A}_\upsilon)
= \{\chi\in\ltwoC\,|\,\sum_{k=1}^\infty \omega_k^2 |\chi_k|^2 <\infty\},
\end{aligned}
$$
and ${\cal A}_\tau: D({\cal A}_\tau) = D({\cal A}_\chi)\times D({\cal A}_\upsilon) \times D(\tilde{\cal A}_\chi) \times D(\tilde{\cal A}_\upsilon)\to X$ is defined by the rule
$$
D({\cal A}_\tau) \ni
\tau = \begin{pmatrix}\chi \\ \upsilon \\ \tilde\chi \\ \tilde \upsilon\end{pmatrix} \mapsto {\cal A}_\tau =
\begin{pmatrix}{\cal A}_\chi\chi \\ {\cal A}_\upsilon\upsilon \\ \tilde{\cal A}_\chi \tilde\chi \\ \tilde{\cal A}_\upsilon \tilde \upsilon\end{pmatrix}\in X.
$$

Combining~\eqref{chi_upsilon_k} and~\eqref{chi_upsilon_tau}, we conclude that the coordinates of each solution $\tau(t)\in X$ of~\eqref{chi_upsilon_tau} can be represented as follows:
\begin{equation}\label{chi_upsilon_exp}
\begin{aligned}
&\chi_k(t) = \chi_k(0)e^{\sigma_k^- t},\; \upsilon_k(t) = \upsilon_k(T)e^{\sigma_{k}^+ (t-T)},\; t\in [0,T],\\
&\tilde\chi_k(t) = \tilde\chi_k(0)e^{\sigma_{-k}^- t},\; \tilde\upsilon_k(t) = \tilde\upsilon_k(T)e^{\sigma_{-k}^+ (t-T)}, \;k\in{\mathbb N}.
\end{aligned}
\end{equation}
Formulas~\eqref{chi_upsilon_exp} allow representing the coordinates of the corresponding solution $a(t)$ of~\eqref{aR_abstract} with the use of transformations~\eqref{chi_upsilon_a}:
\begin{equation}\label{a_exp4}
\begin{pmatrix}a^\zeta_k(t) \\ a^q_k(t)\end{pmatrix}= \left(N_k^\top\right)^{-1}\begin{pmatrix}e^{\sigma_{k}^+ (t-T)} & 0 \\ 0 & e^{\sigma_k^- t} \end{pmatrix} \tilde N_k \begin{pmatrix}a^\zeta_k(0) \\ a^q_k(0) \\ a^\zeta_k(T) \\ a^q_k(T)\end{pmatrix},
\end{equation}
where
$$
\tilde N_k = \begin{pmatrix}0 & 0 & n^k_{11} & n^k_{21} \\
n^k_{12} & n^k_{22} & 0 & 0\end{pmatrix},\;
\begin{pmatrix}n^k_{11} & n^k_{12} \\
n^k_{21} & n^k_{22}
\end{pmatrix}= N_k,\; k\in{\mathbb N},
$$
and the matrices $N_k$ are defined by~\eqref{AMN}.
For further analysis, we rewrite~\eqref{a_exp4} in the form
{
\begin{equation}\label{a_zeta_q_exp}
a^{\zeta,q}_k(t) = g_k^- a^{\zeta,q}_k(0) e^{\sigma_{k}^- t}  +
g_k^+ a^{\zeta,q}_k(T) e^{\sigma_{k}^+ (t-T)},
\end{equation}}
with
{
\begin{equation}\label{g_k}
\begin{aligned}
& a^{\zeta,q}_k(t)=\begin{pmatrix}a^\zeta_k(t) \\ a^q_k(t)\end{pmatrix},\quad
g_k^- = \frac12 \begin{pmatrix} 1 & -\Renu \\ -\frac{1}{\Renu} & 1\end{pmatrix},\quad
g_k^+ = \frac12 \begin{pmatrix} 1 & \Renu \\ \frac{1}{\Renu} & 1\end{pmatrix}.
\end{aligned}
\end{equation}}
Analogously, from~\eqref{chi_upsilon_exp} and~\eqref{chi_upsilon_a}, we obtain the following representation of
{$
a^{z,p}_k(t)=\begin{pmatrix}a^z_k(t) \\ a^p_k(t)\end{pmatrix}
$}
for all $k\in {\mathbb N}$:
{
\begin{equation}\label{a_z_p_exp}
a^{z,p}_k(t)= g_k^- a^{z,p}_k(0) e^{\sigma_{-k}^- t}  +
g_k^+ a^{z,p}_k(T) e^{\sigma_{-k}^+ (t-T)}.
\end{equation}}

{
A straightforward evaluation of the matrix norms for~\eqref{g_k} shows that
\begin{equation}\label{g_norm}
\|g_k^+\| =  \|g_k^-\| = \frac{1 + (\Renu)^2}{2|\Renu|} \lesssim 1/|b_k| \; \text{for all}\; k\in {\mathbb N}.
\end{equation}
By applying the estimates~\eqref{exp_poly_decay}, \eqref{exp_poly_decay_backward}, and~\eqref{g_norm} to~\eqref{a_zeta_q_exp}
we obtain
\begin{equation}\label{a_est_modal}
\Vert a_k^{\zeta,q}(t)\Vert
\leq  \frac{\tilde C_\beta}{|b_k|^{\beta+1}}
\left(\frac{\Vert a_k^{\zeta,q}(0)\Vert}{(t+1)^\beta}
+\frac{\Vert a_k^{\zeta,q}(T)\Vert}{(T-t+1)^\beta}\right),\,t\in[0,T],
\end{equation}
with some constant $\tilde C_\beta>0$, and the same estimate holds for $\Vert a_k^{z,p}(t)\Vert$.
Multiplying the modal estimate~\eqref{a_est_modal} by $|b_k|^{\beta+1}$ and summing the squares over $k$ yields the abstract bound
\begin{equation}\label{a_est}
\norm{a(t)}_{V_{0,\gamma}^X}
\leq \bar C_\beta\left(
\frac{\norm{a(0)}_{X}}{(t+1)^\beta}
+\frac{\norm{a(T)}_{X}}{(T-t+1)^\beta}
\right),
\end{equation}
for all $t\in [0,T]$, with some constant $\bar C_\beta>0$.
By Lemma~\ref{lem:pi0-weighted}, we can transfer the above weighted norms through the bounded isomorphisms $\pi_0$ and $\pi_0^{-1}$ on $V_{0,\gamma}^X$.
Therefore,~\eqref{a_est} yields the required inequality~\eqref{turnpike}.
}
\end{proof}

\subsection{End of the proof of Theorem~\ref{thm_turnpike_coordinates}}
Let $(\hat x,\hat\lambda,\hat\mu,\hat u)$ and $(x(t),\lambda(t),\mu(t),u(t))$ be solutions of Problem~\ref{problem_stat} and Problem~\ref{problem_dyn}, respectively, defined by~\eqref{static_minimizer},~\eqref{cs_op},~\eqref{sys_adjoint},~\eqref{u_extremal},
and let the discrepancy $\Delta(t)$ be defined by~\eqref{delta_notation}.
{
Then
$$
\Vert\Delta(t)\Vert_{V_{0,\gamma}^X}^2 = 2\Big(\Vert x(t)- \hat x\Vert_{V_{0,\gamma}}^2+\Vert \Lambda(t) - \hat \Lambda \Vert_{V_{0,\gamma}}^2\Big),
$$
with $\gamma=-\beta-1$.
Hence, by Proposition~\ref{thm_turnpike},
there exists a constant $\tilde C>0$ such that, for any $T>0$,
\begin{equation}\label{dich}
\Vert x(t)- \hat x\Vert_{V_{0,-\beta-1}}+\Vert \Lambda(t) - \hat \Lambda \Vert_{V_{0,-\beta-1}}
\quad\leq \tilde C\left(\frac{\Vert\Delta(0)\Vert_X}{(t+1)^\beta}+\frac{\Vert\Delta(T)\Vert_X}{(T-t+1)^\beta}\right)\;\text{for all}\; t\in[0,T].
\end{equation}
To conclude the proof of Theorem~\ref{thm_turnpike_coordinates}, we need the following two auxiliary results.}

\begin{lemma}\label{lem_coshbound}
Let $\{\nu_k\}_{k\geq 1}$ be defined in Lemma~\ref{lemma_eigenvalues_asympt}. Then there exist constants $C_{\tanh},C_{\cosh}>0$ such that
\begin{equation}\label{cosh_bound}
\sup_{k\geq 1}\sup_{s\geq 0}|\tanh(\nu_k s)|\leq C_{\tanh},\;
\sup_{k\geq 1}\sup_{s\geq 0}\frac{1}{|\cosh(\nu_k s)|}\leq C_{\cosh}.
\end{equation}
%\textcolor{blue}{(ET: For the sequence $(\nu_k)$: Lemma~\ref{lemma_eigenvalues_asympt} gives $|\mathrm{Im}\,\nu_k|/(\Re \nu_k)=\mathrm{O}(1/\omega_k)$, hence $|\arg(\nu_k)|\leq \pi/4$ for all $k$ large enough; the finitely many remaining indices can be absorbed into the constants by enlarging $C_{\tanh},C_{\cosh}$.)}
\end{lemma}

\begin{proof}
{By Lemma~\ref{lemma_eigenvalues_asympt}, for all $k$ large enough we have $\nu_k=\lambda_k^0+\varepsilon_k$, $\lambda_k^0>0$ and $|\varepsilon_k|<\frac{\lambda_k^0}{2}$. Hence, for such $k$, we have
$\Re \nu_k \geq \lambda_k^0-|\varepsilon_k|>\frac{\lambda_k^0}{2}>0$ and $|\mathrm{Im}\, \nu_k|\leq |\varepsilon_k|<\frac{\lambda_k^0}{2}\leq \Re \nu_k$.
In particular, $\nu_k$ belongs to the closed sector
$S_{\theta_0}=\{z\in\mathbb C\ \mid\ |\arg z|\leq \theta_0\}$ with $\theta_0=\frac{\pi}{4}$,
for all $k\geq k_0$ (some $k_0\in\mathbb N$). For the finitely many indices $1\leq k<k_0$,
we simply increase the constants at the end of the proof; hence it is enough to prove that
\eqref{cosh_bound} holds uniformly for all $z$ in $S_{\theta_0}$.}

Now, fix any $\theta\in(0,\pi/2)$ and consider the sector $S_\theta=\{z\ \mid\ |\arg z|\leq\theta\}$.
We claim that there exist constants $C_{\tanh}(\theta),C_{\cosh}(\theta)>0$ such that
\begin{equation}\label{eq:sector-bounds}
\sup_{z\in S_\theta}|\tanh z|\leq C_{\tanh}(\theta),\qquad
\sup_{z\in S_\theta}\frac{1}{|\cosh z|}\leq C_{\cosh}(\theta).
\end{equation}
Indeed, let $R>0$ and set $q=e^{-2R\cos\theta}\in(0,1)$. For any $z\in S_\theta$ with $|z|\geq R$,
we have $\Re z \geq |z|\cos\theta\geq R\cos\theta$, hence $|e^{-2z}|=e^{-2\Re z}\leq q$.
Using the identities
$$
\tanh z=\frac{1-e^{-2z}}{1+e^{-2z}},
\qquad
\cosh z=\frac{e^{z}}{2}\,\bigl(1+e^{-2z}\bigr),
$$
we obtain, for $|z|\geq R$,
$$
|\tanh z|\leq \frac{1+|e^{-2z}|}{1-|e^{-2z}|}\leq \frac{1+q}{1-q},
$$
and
$$
\frac{1}{|\cosh z|}
\leq \frac{2e^{-\Re z}}{|1+e^{-2z}|}
\leq \frac{2e^{-\Re z}}{1-|e^{-2z}|}
\leq \frac{2e^{-R\cos\theta}}{1-q}.
$$
On the compact set $K_\theta(R)=\{z\in S_\theta\ \mid\ |z|\leq R\}$, the function $\tanh$ is continuous,
hence $M_1=\sup_{z\in K_\theta(R)}|\tanh z|<\infty$.
Moreover, $\cosh z\neq0$ for all $z\in S_\theta$ because the zeros of $\cosh$ are
$\{i(\pi/2+\pi n)\ \mid\ n\in\mathbb Z\}$, which lie on the imaginary axis, and $\theta<\pi/2$.
Therefore $1/\cosh$ is continuous on $K_\theta(R)$ and
$M_2=\sup_{z\in K_\theta(R)}1/|\cosh z|<\infty$.

Combining the bounds on $|z|\leq R$ and $|z|\geq R$, we can set
$$
C_{\tanh}(\theta)=\max\Big(M_1,\frac{1+q}{1-q}\Big),\qquad
C_{\cosh}(\theta)=\max\Big(M_2,\frac{2e^{-R\cos\theta}}{1-q}\Big),
$$
which proves \eqref{eq:sector-bounds}.

\medskip

{Let us now conclude.
Since $\nu_k\in S_{\theta_0}$ for all $k\geq k_0$, we have $\nu_k s\in S_{\theta_0}$ for all $s\geq 0$.
Applying \eqref{eq:sector-bounds} with $\theta=\theta_0$ gives, for all $k\geq k_0$,
$$
\sup_{s\geq 0}|\tanh(\nu_k s)|\leq C_{\tanh}(\theta_0),\;
\sup_{s\geq 0}\frac{1}{|\cosh(\nu_k s)|}\leq C_{\cosh}(\theta_0).
$$
For the finitely many indices $1\leq k<k_0$, each quantity
$\sup_{s\geq 0}|\tanh(\nu_k s)|$ and $\sup_{s\geq 0}1/|\cosh(\nu_k s)|$ is finite because $\Re\nu_k>0$
(no zeros of $\cosh$ on the ray $\nu_k\mathbb R_+$ and exponential growth as $s\to\infty$).
Taking the maximum of these finitely many values with the constants above yields global constants
$C_{\tanh},C_{\cosh}$ such that \eqref{cosh_bound} holds for all $k\geq 1$.}
\end{proof}

\begin{proposition}\label{prop_shoot}
Assume \eqref{A1}, \eqref{A2}, and~\eqref{A3}.
%There exists a constant $\bar C>0$, independent of $T$, such that the associated solution $\Delta$ of the optimality system satisfies
{There exists a constant $\bar C>0$, independent of $T$, such that for every solution $(X,\Lambda,u)$ of Problem~\ref{problem_dyn}, the associated solution $\Delta$ of the optimality system satisfies}%
\footnote{The inequality \eqref{ineq_shoot} is a ``shooting estimate", in the spirit of~\cite[Lemma~2]{trelat2018steady}.}:
\begin{equation}\label{ineq_shoot}
\Vert\Delta(0)\Vert_X+\Vert\Delta(T)\Vert_X
\leq \bar C\Big(\Vert x(0)-\hat x\Vert_{V_{0,1}}+\Vert\hat\lambda\Vert_{\ell^2}+\Vert\hat\mu\Vert_{\ell^2}\Big).
\end{equation}
\end{proposition}

\begin{proof}
We work in the complexified formulation. Set $\theta(t)=\pi_1\Delta(t)=(z(t),\zeta(t),p(t),q(t))\in X$ so that $\dot\theta=\mathcal A\theta$.
We also write $x(t)=(z(t),\zeta(t))$ and $y(t)=(p(t),q(t))$.
The boundary conditions are
$x(0)=x^0-\hat x$ and $y(T)=-(\hat p,\hat q)$,
where $\hat p=\hat\lambda+i\hat\mu$ and $\hat q=\hat\lambda-i\hat\mu$.

First, using Proposition~\ref{thm_Riesz} and the isomorphism $\pi_0$ defined by \eqref{pi0t}, in the coordinates given by $a(t)=\pi_0^{-1}\theta(t)=(a^z(t),a^\zeta(t),a^p(t),a^q(t))\in X$, the dynamics is block diagonal and, modewise, it is governed by the $2\times 2$ matrices $A_k^{z,p}$ and $A_k^{\zeta,q}$ given in \eqref{a_zp}.
Applying the modewise diagonalizers $M_k,N_k$ of~\eqref{AMN} yields the stable-unstable variables $(\chi_k,\upsilon_k,\widetilde\chi_k,\widetilde\upsilon_k)$ introduced in~\eqref{chi_upsilon_a}, propagated by~\eqref{chi_upsilon_exp}.

Collecting the unknown costate traces at $t=0$ into $y_0=(p(0),q(0))\in H\times H$, and the prescribed terminal trace as $y(T)=-(\hat p,\hat q)$, the boundary value problem can be written as a linear operator equation of the form $W_T y_0=r_T$, where $W_T$ is the corresponding shooting operator and $r_T$ depends linearly on $(x(0),y(T))$.

Let us now prove that the shooting operator is uniformly invertible.
In the above coordinates $a(t)$, the shooting operator is diagonal modewise, and its scalar multipliers are combinations of $\tanh(\nu_k T)$ and $1/\cosh(\nu_k T)$ times rational factors in $b_k$ and $\mathrm{Re}(\nu_k)$.
Lemma~\ref{lem_coshbound} gives a constant $C_*>0$, independent of $T$ and $k$, such that
$$
\sup_{k\geq 1}\sup_{T>0}\vert\tanh(\nu_k T)\vert\leq C_*,
\quad
\sup_{k\geq 1}\sup_{T>0}\frac{1}{\vert\cosh(\nu_k T)\vert}\leq C_*.
$$
Using also $\mathrm{Re}(\nu_k)\asymp \vert b_k\vert$ from Lemma~\ref{lemma_eigenvalues_asympt}, we obtain a uniform bound for the inverse shooting map, and thus
$$
\Vert y_0\Vert_{H\times H}\leq C_* \left(\Vert x(0)\Vert_{V_{0,1}\times V_{0,1}}+\Vert y(T)\Vert_{H\times H}\right).
$$

To conclude, let us infer bounds for $\Delta(0)$ and $\Delta(T)$.
Since $y(T)=-(\hat p,\hat q)$, we have
$\Vert y(T)\Vert_{H\times H}\leq 2 (\Vert\hat\lambda\Vert_{\ell^2}+\Vert\hat\mu\Vert_{\ell^2} )$.
The same diagonal representation of the flow gives a bound of the form
$$
\Vert x(T)\Vert_{H\times H}\leq C_* \left( \Vert x(0)\Vert_{V_{0,1}\times V_{0,1}}+\Vert y(T)\Vert_{H\times H}\right),
$$
with the same constant independent of $T$.
Finally, since $\pi_0$ and $\pi_1$ are fixed isomorphisms on $X$, the estimates for $(x(0),y(0))$ and $(x(T),y(T))$ transfer to $\Delta(0)$ and $\Delta(T)$ in the $X$-norm, and this yields~\eqref{ineq_shoot}.
\end{proof}

\medskip

Finally, the assertion of Theorem~\ref{thm_turnpike_coordinates} follows from inequalities~\eqref{dich} and~\eqref{ineq_shoot}. $\square$

\section{Application to a rotating flexible beam}\label{sec_beam}
Consider the Euler--Bernoulli beam of length $l$ attached to a rotating rigid body (hub).
The transverse displacement $w(x,t)$ of the beam at $x\in [0,l]$ with respect to the reference frame, fixed on the hub, is described by the following partial differential equation~(see, e.g.,~\cite[Chap.~4]{luo1999stability}, \cite[Chap.~3]{zuyev2015partial}):
\begin{equation}
\frac{\partial^2 w(x,t)}{\partial t^2} + c^2\frac{\partial^4
w(x,t)}{\partial x^4} = - (x+d)u(t),\quad x\in [0,l],\; t\geq 0,
\label{PDE}
\end{equation}
and the boundary conditions
\begin{equation}
w\vert_{x=0} =0,\; \left.\frac{\partial w}{\partial x}\right\vert_{x=0} =0,\;\left.\frac{\partial^2 w}{\partial x^2}\right\vert_{x=l} =0,\;\left.\frac{\partial^3 w}{\partial x^3}\right\vert_{x=l} =0.
\label{BC}
\end{equation}
The control $u(t) \in\mathbb R$ corresponds to the angular acceleration of the hub, $d>0$ is the radius of the hub, $c= \sqrt{EI/\rho}>0$,
$E$ is the Young's modulus, $I$ is the second moment of cross-section area, and $\rho$ is the mass per unit length of the beam.
The above control system corresponds to the case of a ``slow'' rotation of the hub, i.e. the higher-order terms are omitted in~\eqref{PDE}--\eqref{BC}.

Let the beam deflection be expanded as
\begin{equation}
 w(x,t) = \sum_{n=1}^{\infty} W_{n}(x)q_{n}(t),\quad x\in [0,l],
\label{modexp}
 \end{equation}
 where the $W_n$ are eigenfunctions of the spectral problem
 \begin{equation}
 \frac{d^4}{dx^4}W_{n}(x) = \lambda_{n}
 W_{n}(x),\;x\in [0,l],
 \label{SL_problem_in}
\end{equation}
\begin{equation}
 W_{n}(0)=W_{n}'(0)=W_{n}''(l)=W_{n}'''(l)=0.
 \label{SL_BC_in}
\end{equation}
The eigenvalues $\lambda_n$ and normalized eigenfunctions $W_n(x)$ of~\eqref{SL_problem_in}--\eqref{SL_BC_in} are expressed as~\cite[Chap.~3]{zuyev2015partial}:
\begin{equation}
\lambda_n = (\delta_n / l)^4,\qquad W_{n}(x)=k_{n}
\phi_n(x/l),\qquad
k_n = \pm {l}^{-1/2} \left( \|\phi_n\|_{L^2(0,1)} \right)^{-1},
\label{lambda_W}
\end{equation}
where
$$
 \phi_n(x)= -\frac{1+\gamma_n}{2} e^{\delta_n x} - \frac{1-\gamma_n}{2} e^{-\delta_n x} +\gamma_n\sin(\delta_n x)+\cos(\delta_n x),
$$
 \begin{equation}
 \gamma_n=-\frac{e^{\delta_n}-\sin\delta_n+\cos\delta_n}{e^{\delta_n}+\sin\delta_n+\cos\delta_n}<0,
 \label{phin}
\end{equation}
and $ \delta_1 <\delta_2<...<\delta_n<...$ are positive solutions of
 \begin{equation}
 1+\cos(\delta_n)\cosh(\delta_n) = 0.
 \label{TransEqBeta}
 \end{equation}
The eigenfunctions $\{W_n(x)\}_{n=1}^\infty$ defined by~\eqref{lambda_W} form an orthonormal basis in $L^2(0,l)$,
and $\|\phi_n\|_{L^2(0,1)}=O(1)$ as $n\to\infty$ (see~\cite[Lemma~4.6]{luo1999stability}).  It can also be seen from~\eqref{TransEqBeta} that
\begin{equation}
\delta_n = \frac{\pi (2n-1)}{2} + o\left(\frac{1}{n}\right)\quad \text{as}\;\;n\to\infty.
\label{beta_asympt}
\end{equation}

By substituting~\eqref{modexp} into~\eqref{PDE}--\eqref{BC} and performing the change of variables
$$
 \xi_n(t) = \omega_n q_n(t),\; \eta_n(t) = \dot q_n(t),
 $$
we obtain the system~\eqref{cs_coord} with
 \begin{equation}
 \begin{aligned}
 & \omega_n = \frac{c \delta_n^2}{l^2}=\frac{c\pi^2(2n-1)^2}{4l^2}+o(1),\\
 & b_n = 2 l k_{n} {\delta_n}^{-1}\left(l
{\delta_n}^{-1}-d \gamma_n\right) =O(1/n)>0\quad\text{as}\;\;n\to \infty.
\end{aligned}
\label{omega_beta}
 \end{equation}

The above representations imply that $\omega_n$ and $b_n$ satisfy Assumptions~\eqref{A1}--\eqref{A3}
(Assumption~\eqref{A3} follows from~\cite[Lemma~4.3]{zuyev2015partial}).

Exploiting~\eqref{omega_beta}, we conclude that the sequence $\{b_n \omega_n^\alpha\}$ is asymptotically equivalent to $\{n^{2\alpha-1}\}$ as $n\to \infty$. {Therefore, condition~\eqref{A6} holds for any $\alpha\geq \frac{1}{2}$.
In view of Remark~\ref{rem_paramasymptotics}, to satisfy~\eqref{A5}, we select $\rho$ from inequality~\eqref{alpha_rho_p} with $p=2$.
Thus, any $\rho\in (1,\frac32)$ satisfies~\eqref{A5}.}
{\em By Theorem~\ref{thm_turnpike_coordinates}, the solutions of Problems~\ref{problem_stat} and~\ref{problem_dyn} satisfy the turnpike property~\eqref{eq:main_turnpike} with {any $\beta=\rho-1\in (0,\frac{1}{2})$}, provided that the parameters of control system~\eqref{cs_coord} correspond to the rotating beam model (case~\eqref{omega_beta}).}

\section{Conclusion}
The main contribution of this paper is the establishment of a previously unreported polynomial turnpike phenomenon in linear-quadratic optimal control problems for a class of oscillating infinite-dimensional systems.

Compared with the classical exponential turnpike theory for uniformly exponentially stabilizable systems, the key novelty here is that the closed-loop spectral gap collapses at high frequencies, which prevents uniform exponential estimates in the natural energy space.
The analysis therefore relies on a polynomial dichotomy and on weighted norms: the turnpike estimate is formulated in the weak topology $V_{0,-\beta-1}$, which compensates the mode-dependent loss intrinsic to bounds of the form $e^{-\vert b_k\vert t}\le C_\beta\vert b_k\vert^{-\beta}(t+1)^{-\beta}$.
From a methodological viewpoint, the proof combines a Riesz basis construction for the Hamiltonian generator of the variational system with uniform wedge estimates for the associated hyperbolic functions and a weighted Bari-type argument.

Note that the pointwise control discrepancy term $|u(t)-\hat u|$ does not appear on the left-hand side of the main inequality~\eqref{eq:main_turnpike}.
Since the optimal control $u(t)$ is well-defined by formula~\eqref{u_extremal} under the proven property $\Lambda(t)\in H$ for the costate, we have $|u(t)-\hat u|\leq \Vert b\Vert_{\ell^2}\,\Vert\mu(t)-\hat\mu\Vert_{\ell^2}$.
Thus, any \emph{strong} turnpike estimate controlling $\Vert\mu(t)-\hat\mu\Vert_{\ell^2}$ would imply the same decay rate for $|u(t)-\hat u|$.

A natural next step is to sharpen the topology on the left-hand side.
A pointwise polynomial turnpike estimate in the strong energy space $H$ would immediately provide decay information for the control discrepancy through the feedback relation, but such an upgrade seems to require additional endpoint regularity or more restrictive structural assumptions.
Another natural question is the sharpness of the decay exponent and of the chosen weighted topology.
It would be interesting to relate the best possible polynomial rate and the weakest meaningful topology to resolvent growth conditions in the spirit of polynomial stability theory, see, e.g.,~\cite{batty2008nonuniform,borichev2010optimal,batty2016fine}.

The present contribution is restricted to a free-end optimal control problem.
The consideration of a final-point constraint $x(T)=x^1$ is associated with serious additional challenges in the regularity analysis of the costate.
In the present weak topology $V_{0,-\beta-1}$, a pointwise turnpike estimate would require additional endpoint regularity and more delicate spectral arguments.
We therefore leave the fixed-endpoint case as an open problem for future research.

Beyond the fixed-endpoint issue, several further directions seem natural:
\begin{itemize}
\item extend the analysis to multi-input or boundary controlled hyperbolic systems,
where the spectral structure may not
{satisfy gap conditions of the form~\eqref{A2},~\eqref{A3}, or may include multiplicities
(e.g., in the rotating Timoshenko beam model~\cite{zuyev2007stabilization},~\cite[Theorem~5.1]{zuyev2015partial} or the Kirchhoff plate~\cite[Chap.~6]{zuyev2015partial})};
\item consider semilinear perturbations, for which only local or approximate turnpike properties may hold;
\item investigate turnpike phenomena around non-stationary sets (e.g., periodic optimal regimes), when steady states are not optimal.
\end{itemize}

\section*{Acknowledgment}

Alexander Zuyev gratefully acknowledges the funding by the
European Regional Development Fund (ERDF) within the programme Research
and Innovation — Grant Number ZS/2023/12/182138.

\small
\bibliographystyle{abbrv}
\bibliography{bib_turnpike}

\end{document}